\documentclass[microtype]{gtpart}
\usepackage{latexsym}
\usepackage{amsmath,amsthm,amsfonts,amssymb}
\usepackage[T1]{fontenc}
\usepackage[latin1]{inputenc}
\usepackage{aeguill}
\usepackage{url}
\usepackage{enumerate}
\usepackage{mdwlist}
\usepackage[a4paper,textwidth=15cm,textheight=25cm]{geometry}
\usepackage{xspace,color}
\usepackage{export}
\usepackage{hyperref}

\usepackage{pinlabel}
\usepackage[all]{xy}

\title{Residual properties of automorphism groups   of (relatively) hyperbolic groups}

\author{Gilbert Levitt}
\givenname{Gilbert}
\surname{Levitt}
\address{Laboratoire de Math\'ematiques Nicolas Oresme\\
Universit\'e de Caen et CNRS (UMR 6139)\\
\newline BP 5186\\
14032 CAEN Cedex 5\\
France}
\email{levitt@unicaen.fr}

\author{Ashot Minasyan}
\address{Mathematical Sciences \\ University of Southampton
\\ Highfiled Southampton SO17 1BJ \\ UK}
\email{aminasyan@gmail.com}


\newtheorem{thm}{Theorem}[section]
\newtheorem{cor}[thm]{Corollary}
\newtheorem{prop}[thm]{Proposition}
\newtheorem{lem}[thm]{Lemma}
\newtheorem{df}[thm]{Definition}
\newtheorem{as}[thm]{Assumption}
\theoremstyle{remark}

\newtheorem*{algo*}{Algorithm}
\newtheorem*{rem*}{Remark}
\newtheorem{rem}[thm]{Remark}
\newtheorem*{example*}{Example}
\newtheorem{example}[thm]{Example}

\newcommand{\es}{\emptyset}
\newcommand{\m} {^{-1}}
\newcommand {\calb} {{\mathcal {B}}}
\newcommand {\calc} {{\mathcal {C}}}
\newcommand {\cale} {{\mathcal {E}}}
\newcommand {\calh} {{\mathcal {H}}}
\newcommand {\calp} {{\mathcal {P}}}
\newcommand {\calq} {{\mathcal {Q}}}
\newcommand {\calt} {{\mathcal {T}}}
\newcommand{\Out} {{\mathrm{Out}}}
\newcommand{\Aut} {{\mathrm{Aut}}}
\newcommand{\Inn} {{\mathrm{Inn}}}
\newcommand {\N} {{\mathbb {N}}}
\renewcommand {\Z} {{\mathbb {Z}}}
\newcommand{\inc}{\subset}
\newcommand {\Hy} {{\mathbb {H}}}
\renewcommand{\O}{\mathcal{O}}
\newcommand{\G}{G}
\newcommand{\llangle}{\langle\hspace{-2.5pt}\langle}
\newcommand{\rrangle}{\rangle\hspace{-2.5pt}\rangle}
\newcommand{\Inc}{\mathrm{Inc}}

\begin{document}
\begin{abstract}
We show that $\Out(G)$ is residually finite if $G$ is one-ended and hyperbolic relative to virtually polycyclic subgroups. More generally, if $G$ is one-ended and hyperbolic relative to proper residually finite subgroups,  the group of outer automorphisms preserving the peripheral structure is
residually finite. We also show that $\Out(G)$ is virtually residually $p$-finite for every prime $p$ if $G$ is one-ended and  toral relatively hyperbolic, or infinitely-ended and virtually residually $p$-finite.
\end{abstract}

\maketitle

\section{Introduction}
A group $G$  is \emph{residually finite} if, given any $g\ne1$, there exists a homomorphism $\varphi$ from $G $ to a finite group such that $\varphi(g)\ne 1$.
Residual finiteness is an important property of groups.   It is equivalent to the statement that $G$ embeds into its profinite completion. Well known theorems of Mal'cev claim that finitely generated residually finite groups
are Hopfian, and finitely presented residually finite groups have solvable word problem.
Many groups are known to be residually finite (in particular,  finitely generated linear groups), but it is a big open question whether all (Gromov) hyperbolic groups are residually finite.

It is a standard and classical fact \cite {GBaumslag-63} that the automorphism group $\Aut(G)$ is residually finite if $G$ is finitely generated  and residually finite,
but this is not true for the outer automorphism group $\Out(G)$. Indeed, any finitely presented group may be represented as $\Out(G)$ with $G$ finitely generated  and  residually finite \cite{BW}.

A special case of our main theorem is:

\begin{cor}\label{main1}
If $G$ is  one-ended, and hyperbolic relative to a   family $\calp=\{P_1,\dots, P_k\}$ of virtually polycyclic groups, then $\Out(G)$ is residually finite.
\end{cor}

This is new even if $G$ is a  torsion-free hyperbolic group. Work by Sela implies that a finite index subgroup of $\Out(G)$ is   virtually a central extension of
a free abelian group by a direct product of mapping class groups (see \cite{GD}). Though mapping class groups are known to be residually finite
following work by Grossman based on conjugacy separability \cite{A-K-T, Grossman}, this is not enough to deduce residual finiteness of $\Out(G)$
because the extension may
fail to split (i.e., be a direct product); see the example discussed  in the introduction of \cite{GD}.

To complement Corollary \ref{main1}, recall that $\Out(G)$ is residually finite if $G$ is residually finite and has infinitely many ends \cite{M-O-norm_aut}.
On the other hand,  $\Out(G*F_2)$ contains $G$ (with $F_2$   the free group of rank $2$),
 so it is not residually finite if $G$ is not. Thus
one-endedness  cannot be dispensed with in Corollary \ref{main1}.   This also gives a direct way of proving the following fact, which may otherwise be obtained by combining small cancellation
theory over hyperbolic groups with the results from \cite{M-O-norm_aut}.

\medskip
\begin{prop}\label{prop:equiv}  The following are equivalent:
\begin{itemize}
  \item[(i)] Every hyperbolic group is residually finite.

  \item[(ii)] For every hyperbolic group $G$, the group $\Out(G)$ is residually finite.
\end{itemize}
\end{prop}

Virtual polycyclicity of the $P_i$'s is used in two ways in Corollary \ref{main1} (see  Subsection~\ref{pfcor}).   It ensures that $P_i$ is residually finite, and also  that automorphisms of
$G$ respect the peripheral
structure: $P_i$ is mapped to a conjugate of a $P_j$ (this only holds if no $P_i$ is virtually cyclic, but such a restriction causes no loss of generality, see Subsection \ref{defrhyp}).
In fact, the peripheral  structure is preserved  if every $P_i$ is small
(i.e., it does not contain the free group $F_2$) or, more generally,  NRH
(non relatively hyperbolic) -- see \cite{M-O-fixed_sbgp}
for a proof and a list of  examples of NRH groups.

 Since the peripheral structure is not always preserved,
we restrict to the subgroup $\Out(G;\calp)$ of $\Out(G)$ defined  as the group of classes of automorphisms mapping each $P_i$ to a conjugate.

 \begin{thm}
 \label{gener}
 Let $G$ be a group hyperbolic relative to  a family of proper finitely generated subgroups $\calp=\{P_1,\dots,P_k\}$.
 If $G$ is one-ended relative to $\calp$, and every $P_i$ is residually finite, then $\Out(G; \calp)$ is residually finite.
 \end{thm}

Being \emph{one-ended relative to $\calp$} means that  $G$ does not split over a finite group with each $P_i$ contained in a vertex group (up to conjugacy). This is   weaker than having at most one end.

 If every $P_i$ is NRH, then $\Out(G; \calp)$ has finite index in $\Out(G)$, so  we deduce residual finiteness of $\Out(G)$.

The following example shows that it is indeed necessary to assume that all peripheral
subgroups $P_i$ are residually finite  in Theorem \ref{gener}.

 \begin{example}
Let $H$ be a torsion-free non-residually finite group  with trivial center (such as the Baumslag-Solitar group $BS(2,3)$). Let $K$ be a one-ended torsion-free
hyperbolic group (e.g., the fundamental group of a closed
hyperbolic surface), and let $\langle k\rangle\leqslant K$ be a maximal cyclic subgroup.
Let $G$ be the amalgamated product $(H\times\langle k\rangle)*_{\langle k\rangle}K$.
Then $G$ is one-ended, torsion-free, and hyperbolic relative to $\calp=\{H\times \langle k \rangle\}$ (see \cite{Osin-comb}). Non-trivial elements   $h \in H$ define twist
automorphisms (they act as conjugation by $h$ on $H\times \langle k \rangle$
and trivially on $K$), which give rise to non-trivial outer automorphisms in $\Out(G;\calp)$ because $H$ has trivial center.
Thus $H$ can be embedded into $\Out(G;\calp)$, and so $\Out(G;\calp)$ is not residually finite.
\end{example}

In the last section of the paper we consider residual $p$-finiteness. If $p$ is a prime,   $G$  is \emph{residually $p$-finite} if, given any   non-trivial element $g \in G$,
there exists a homomorphism $\varphi$ from $G $ to a finite $p$-group such that $\varphi(g)\ne 1$.
A group is \emph{virtually  residually $p$-finite} if some finite index subgroup is residually $p$-finite. Evidently residual $p$-finiteness implies residual nilpotence.
And if a group is virtually residually $p$-finite for at least two distinct primes $p$, then it is virtually torsion-free.

It is well-known that free groups are residually $p$-finite for any prime $p$, and
it is a classical result that a finitely generated linear group is residually $p$-finite for all but finitely many $p$'s (cf. \cite{Wehr-book}).
Lubotzky \cite{Lub} proved that for a finitely generated virtually residually $p$-finite group $G$, the group $\Aut(G)$ is also virtually residually $p$-finite, which is a natural analogue of Baumslag's result \cite{GBaumslag-63}.
Another theorem from \cite{Lub} states that, if $F$ is a free group of finite rank, then $\Out(F)$ is virtually residually $p$-finite for any prime $p$. The latter result was later extended by Paris \cite{Paris} to
fundamental groups of compact oriented surfaces. The next two theorems generalize these results to certain relatively hyperbolic groups:

 \begin{thm} \label{resp}
If $G$ is a one-ended toral relatively hyperbolic group, then $\Out(G)$ is virtually residually $p$-finite for every
prime number $p$.
\end{thm}

Recall that $G$ is called \emph{toral relatively hyperbolic} if it is torsion-free and hyperbolic relative to a finite family of finitely generated abelian groups.
The theorem also applies to groups containing a one-ended toral relatively hyperbolic group with finite index, in particular to virtually torsion-free hyperbolic groups (see Theorem \ref{thm:virt_rp}).

The following statement is a counterpart of Theorem \ref{resp} in the case when $G$ has infinitely many ends, giving an `outer' version of Lubotzky's result \cite{Lub} mentioned above. It
is a natural pro-$p$ analogue of \cite[Thm.\ 1.5]{M-O-norm_aut}.
\begin{thm}\label{thm:resp-inf_ends}
If $G$ is a finitely generated group with infinitely many ends and $G$ is virtually residually $p$-finite for some prime number $p$, then $\Out(G)$ is virtually residually $p$-finite.
\end{thm}

Recall that limit groups (finitely generated fully residually free groups)  are toral relatively hyperbolic    \cite{Alib,Dah}.
Residual  finiteness of $\Out(G)$ for such a group $G$ was proved by Metafsis and Sykiotis in \cite{Met-Syk}.
Combining Theorems \ref{resp} and \ref{thm:resp-inf_ends} gives the following enhancement:

\begin{cor}\label{cor:limit}   If $G$ is a limit group (finitely generated fully residually free group), then $\Out(G)$ is  virtually residually $p$-finite for any prime $p$.
\end{cor}

We start the paper by giving a rather quick proof of Corollary \ref{main1} when $G$ is (virtually) torsion-free hyperbolic, using Sela's description of
$\Out(G)$ recalled above as a starting point.
The proof of Theorem \ref{resp} (given in Section \ref{rep}) uses similar arguments, but in order to prove Theorem \ref{gener}, we have to use   different techniques.

Say that a subgroup of   a relatively hyperbolic group $G$ is  \emph{elementary} if it is virtually cyclic or parabolic (conjugate to a subgroup of some $P_i$).
As in \cite{GL6}, we consider the canonical JSJ decomposition of $G$ over elementary subgroups relative to $\calp$. This is a graph of groups decomposition $\Gamma$
of $G$ such that edge groups are elementary, each $P_i$ is conjugate to a subgroup of a vertex group, and $\Gamma$ is $\Out(G,\calp)$-invariant;
moreover, vertex groups are either elementary, or quadratically hanging (QH), or rigid.

In the first step of the proof of Theorem \ref{gener}  (Section \ref{sec:rid}), we
replace each rigid vertex group by a new group which is residually finite and has infinitely many ends.
In the second step, we make elementary vertex groups, and edge groups,   finite (using residual finiteness of the $P_i$'s).  Apart from simple cases,
the new graph of groups represents a residually finite group $G''$ with infinitely many ends, and so $\Out(G'')$ is residually finite by  \cite {M-O-norm_aut}.
   Thus we get a homomorphism $\Out(G;\calp) \to \Out(G'')$ and we show that such homomorphisms ``approximate'' $\Out(G;\calp)$.

The second step   is easier when $G$ is torsion-free (see Section \ref{tfrh}). Torsion brings technical complications, so in its presence we prefer to
give a different argument using Dehn fillings \cite{Osin-CEP} and Grossman's method \cite{Grossman}. Sections \ref{gtor} and \ref{rep}  are independent of Section \ref{tfrh}.

\medskip
 \textbf{Acknowledgments.}
We wish to thank the organizers of the 2009 Geometric group theory conference in Bedlewo, where this research was started.
 The first author  was   supported in part by ANR-07-BLAN-0141-01 and ANR-2010-BLAN-116-03.   The work of the second author was partially supported by the EPSRC grant EP/H032428/1.

 \section{Notations and residual finiteness}
  First let us specify some notation.

  If $G$ is a group, we denote its center by $Z(G)$. If $H \leqslant G$ is a subgroup, then
  $Z_G(H)$ is its centralizer in $G$.   We will write $|G:H|$ for the index of a subgroup $H$ in $G$.
  For any $g \in G$, we denote by $\tau_g \in \Aut(G)$ the inner automorphism given by  $\tau_g:x\mapsto gxg\m$ for all $x \in G$.

If  $R \subseteq G$, then $\llangle R \rrangle^G$ will denote the normal closure of $R$ in $G$.
   If $A$ is an abelian group, and $n\ge 1$, we will write $nA =\{g^n \mid g \in A\}$  for the corresponding verbal subgroup   of  $A$.

Given $\alpha\in\Aut(G)$, we write $\hat \alpha$ for its image in $\Out(G)$.
We denote by $\Aut(G;\calp)\leqslant\Aut(G)$ the subgroup consisting of automorphisms mapping each $P_i$ to a conjugate, and by $\Out(G;\calp)$ its image in $\Out(G)$.
If every $P_i$ is NRH (e.g., if $P_i$ is not virtually cyclic and has no non-abelian free subgroups), then
$\Out(G;\calp)$ has finite index in $\Out(G)$ (see \cite[Lemma 3.2]{M-O-fixed_sbgp}).

  Given a group $G$, the cosets of finite index  normal subgroups define a basis of the \emph{profinite topology} on $G$. This topology is Hausdorff if and only if $G$ is residually finite.
  A subset $S$ of $G$ is said to be \emph{separable} if $S$ is closed in the profinite topology.
Thus,
if $G$ is residually finite, then any finite subset of $G$ is separable.

A subgroup $K \leqslant G$ is closed in the profinite topology if and only if $K$ is the intersection of a family of finite index subgroups.
It is easy to see that a normal subgroup $N \lhd G$ is separable if and only if $G/N$ is residually finite.
In particular,   if $G$ is residually finite and $N \lhd G$ is finite, then $G/N$ is also  residually finite.

  If $H\leqslant G$ has finite index, then $G$ is residually finite if and only if $H$ is.
We will also use the following fact: the fundamental group of a finite graph of groups with residually finite vertex groups and finite edge groups is residually finite (see,   for instance, \cite[II.2.6.12]{Serre}).

  Recall that in a finitely generated group $G$ every finite index subgroup $K \leqslant G$ contains a finite index subgroup $N$ which is characteristic in $G$, e.g., one can take  $N=\bigcap_{\alpha \in \Aut(G)} \alpha(K)$.
Thus, if $G$  finitely generated and residually finite, then for every $g \in G \setminus\{1\}$ there is a characteristic subgroup $N$ of
finite index in $G$ such that $g \notin N$.

\section{Torsion-free hyperbolic groups}
\label{hyp}

The goal of this section is to   give a short proof of the following statement:

\begin{thm}\label{rfhyp}
Let $G$ be  a one-ended  hyperbolic group. If $G$ is virtually torsion-free, then $Out(G)$ is residually finite.
\end{thm}

\subsection{Automorphisms with twistors } \label{relaut}

Let $H$ be a group.
Fix finitely many subgroups $C_1, \dots, C_s$ (not necessarily distinct),  with $s\ge1$.
We define groups $PMCG(H)$ and $PMCG^\partial(H)$ as in Section 4    of \cite{GD}.

First,  $PMCG(H)$ is the subgroup of
 $\Out(H)$ consisting of all elements $\hat\alpha$ represented by automorphisms $\alpha\in \Aut(H)$ acting on each $C_i$ as $\tau_{a_i}$ for  some   $a_i\in H$.

Let $\Aut^\partial(H)$ be the subset of $\Aut(H) \times H^s$ given by
\begin{multline*}
\Aut^\partial(H)=\{(\alpha;a_1,\dots,a_s) \mid \alpha \in \Aut(H), \, a_i \in H, \, \alpha(c)=a_ic a_i ^{-1} \\ \mbox{ for  all } c \in C_i \mbox{ and   }i=1,\dots,s \}.
\end{multline*}
It is a group, with multiplication   defined
by $$(\alpha;a_1,\dots,a_s) (\alpha';a'_1,\dots,a'_s)=(\alpha\circ\alpha';\alpha(a'_1)a_1,\dots,\alpha(a'_s)a_s),$$
in accordance with the fact that $\alpha\circ\alpha'$ acts on $C_i$ as conjugation by $\alpha(a'_i)a_i$.

One easily checks that $I=\{(\tau_h; h,\dots,h) \in \Aut^\partial(H) \mid h \in H\}$ is a normal subgroup of $\Aut^\partial(H)$.
As in \cite{GD}, we define $PMCG^\partial(H)$ to be the quotient of $\Aut^\partial(H)$ by $I$. Thus
an element of $PMCG^\partial(H)$ is represented by   $(\alpha;a_1,\dots,a_s)$,
where $\alpha$ is an automorphism of $H$ acting on $C_i$ as $\tau_{a_i}$, with $a_i\in H$.  Representatives are defined only up to multiplication
by elements of the form $(\tau_h;h,\dots,h)$; in particular, for each $i$, there is a unique representative with $a_i=1$.
Mapping $(\alpha;a_1,\dots,a_s)$ to $\hat\alpha$ defines a projection $\pi:PMCG^\partial(H)  {\to} PMCG(H)$.

As observed in \cite[Lemma 4.1]{GD} there is a short exact sequence $$\{1\} \to \calt_H \to PMCG^\partial(H) \stackrel{\pi}{\to} PMCG(H)\to \{1\} $$
whose kernel $\calt_H$ is
the \emph{group of twists}. It fits in an exact sequence
$$\{1\} \to Z(H) \to \prod_{i=1}^s Z_{H}(C_i)    \to \calt_H \to \{1\},$$
where the first map is the diagonal embedding and the second map takes $(z_1,\dots,z_s)$ to the class of $(id_H;z_1,\dots,z_s)$ in $PMCG^\partial(H)$.

     If $N\lhd H$ is a normal subgroup invariant under $PMCG (H)$, there are natural homomorphisms $PMCG (H)\to PMCG (H/N)$ and $PMCG^\partial(H)\to PMCG^\partial(H/N)$,
 where the target groups are defined with respect to the images of $C_1,\dots, C_s$ in $H/N$.

  \begin{lem}\label{rfd}
   If $H$ is finitely generated and residually finite, then $PMCG^\partial(H)$ is residually finite.
  \end{lem}

\begin{proof}
Given any non-trivial $\Phi\in PMCG ^\partial(H)$, we can construct a characteristic finite index subgroup $N\lhd H$ such that   $\Phi$ maps non-trivially to the finite group $PMCG ^\partial(H/N)$. Indeed, let
$(\alpha; 1,\dots,a_s)$ be the representative of $\Phi$ with $a_1=1$.  If $\alpha(h)\ne h$ for some $h\in H$, we choose such an $N$ with $h\m\alpha(h)\notin N$.
On the other hand, if $a_i\ne1$ for some $i\ge2$, we choose a characteristic subgroup of finite index $N\lhd H$ with $a_i\notin N$.
\end{proof}

\begin{rem}\label{rem:emb_of_aut}
There are injective homomorphisms   $$ \Aut^\partial(H) \to H^s \rtimes \Aut(H) \mbox{  and }PMCG^\partial(H) \to  H^{s-1} \rtimes\Aut(H)$$ defined by
$$ (\alpha; a_1,a_2,\dots,a_s) \mapsto\left( (a_1\m,a_2\m,\dots,a_s\m),\alpha\right) \mbox{ and } (\alpha; 1,a_2,\dots,a_s)\mapsto  \left( (a_2\m,\dots,a_s\m),\alpha \right)$$ respectively,
with $\Aut(H)$ acting on $H^s$ and $H^{s-1}$ diagonally.
This yields another way of proving Lemma \ref{rfd},   using residual finiteness of
  the semidirect product of a
   finitely generated residually finite group and a residually finite group.
 \end{rem}

\subsection{Surfaces}\label{surf}

We now specialize to the case when $H$ is
the fundamental group of a compact  (possibly non-orientable) surface $\Sigma$ with    boundary components $\calc_1,\dots, \calc_s$; we require $s\ge1$ and $ \chi(\Sigma)<0$. We fix a representative
$C_i$ of $\pi_1(\calc_i)$ in $G$ and a generator $c_i$ of $C_i$. Then $PMCG^\partial(H)$, as defined above, may be identified with the group of isotopy classes of
homeomorphisms of $\Sigma$ equal to the identity on the boundary. In this definition,
the isotopy is  relative to the boundary, so $PMCG^\partial(H)$ contains Dehn twists near boundary components.
If we do not require isotopies to be relative to the boundary, we get  $PMCG (H)$.

 There is
a central extension
$$\{1\}\to \calt_H \to PMCG^\partial(H)\to PMCG (H)\to\{1\}
$$
as above, where $\calt _H\cong \Z^s$  is   generated by   Dehn  twists  near  boundary components of $\Sigma$.
 The inclusion from  $\Z^s$
 to $PMCG^\partial(H)$ may be written algebraically as $$(n_1,\dots,n_s)\mapsto (id_H;c_1^{n_1},\dots, c_s^{n_s}).$$

\begin{lem}\label{omcg}
Let $n\calt_H\lhd \calt  _H$ be the subgroup generated by the $n$-th powers of the twists. Then $PMCG^\partial(H)/n\calt_H $ is residually finite for all sufficiently large  $n\in\N$.
\end{lem}

It is worth noting that residual finiteness does not follow directly from the fact that the group $PMCG^\partial(H)/n\calt_H $ maps onto the residually finite group $PMCG (H)$ with finite kernel.

 \begin{proof}
Let $\Sigma_n$ be the closed orbifold obtained by replacing each boundary component of $\Sigma$ by a conical point of order $n$, and let $O_n=H/\llangle c_1^n,\dots,c_s^n\rrangle^H$ be its fundamental group.

The Euler characteristic of $\Sigma_n$ is $\chi(\Sigma_n)=\chi(\Sigma)+\frac sn$ (see \cite{Scott} or \cite{Th}). It is negative for $n$ large
since  $\chi(\Sigma)<0$,
so   $\Sigma_n$ is a hyperbolic
orbifold (see \cite[Thm.\ 13.3.6]{Th}).
It follows that $O_n$ embeds into the group of isometries of
the hyperbolic plane as a non-elementary subgroup.
In particular, $O_n$ has trivial center and is residually finite.

Defining  $\calt_{O_n}$ as the kernel of the map $PMCG^\partial(O_n)\to PMCG (O_n)$,
there is a commutative diagram of short exact   sequences

 $$
 \xymatrix{
\{1\}\ar[r]&\calt_H\ar[r]\ar[d]^\theta&PMCG^\partial(H)\ar[r]\ar[d]& PMCG (H)\ar[r]\ar[d]&\{1\}\\
\{1\}\ar[r]&\calt_{O_n}\ar[r]& PMCG^\partial(O_n)\ar[r]& PMCG (O_n)\ar[r]& \{1\}.
}
$$

Since $O_n$ has trivial center, and the image of $C_i$ in $O_n$ is equal to  its centralizer, $\calt_{O_n}$ is isomorphic to $(\Z/n\Z)^s$,  so the kernel of  the map $\theta$ from $\calt _H$ to $\calt_{O_n}$ is  precisely $n\calt _H$ (it is not bigger).
The maps from $PMCG^\partial(H)$ to $PMCG (H)$ and $PMCG^\partial(O_n)$ both factor through $PMCG^\partial(H)/n\calt _H$,
and the intersection of their kernels is $\ker\theta=n\calt _H$. In other words, any non-trivial element of $PMCG^\partial(H)/n\calt _H$ has a
non-trivial image in $PMCG (H)$ or in $PMCG^\partial(O_n)$. It is well-known that $PMCG (H)$ is residually finite (it is contained in $\Out(H)$,
which is residually finite  by \cite{Grossman} because $H$  is a finitely generated free group),
and
$PMCG ^\partial(O_n)$ is residually finite by Lemma \ref{rfd}, so $PMCG^\partial(H)/n\calt _H$ is residually finite.
\end{proof}

\begin{rem} It follows from the classification of non-hyperbolic $2$-orbifolds \cite{Th} that $n\ge 3$ is always sufficient  in Lemma \ref{omcg}.
\end{rem}

\subsection{An algebraic lemma}

\begin{lem}\label{alg}
Consider a finite set $V$ and groups $P_v$, $v \in V$, with normal subgroups $T_v$ free abelian of finite ranks. Let $P=\prod_{v \in V} P_v$ and $  T=\prod_{v \in V} T_v \leqslant P$ be their direct products.
  Note that $nT_v$ is characteristic in $T_v$, hence it is normal in $P_v$.

If
 $P_v/nT_v$ is residually finite for every $v \in V$ and for all sufficiently large $n\in \mathbb{N}$,  then any subgroup   $Z\leqslant T $ is closed in the profinite topology of $P $.
 In particular,  if $Z$ is normal in $P$, then $P/Z$ is residually finite.
 \end{lem}

 \begin{proof}
Let us first prove the result when $Z$ has finite index $k$ in the free abelian group $T$.
It  contains $ nT$ (with finite index) whenever $k$ divides $n$. For $n$ large,  $nT_v$ is separable in $P_v$ for every $v \in V$ (because $P_v/nT_v$ is residually finite),
so  $nT=\prod_{v \in V} nT_v$ is separable in $P$, by the properties of direct products. We deduce that
 $Z$ is closed in $P$, because it is equal to a finite union of cosets modulo $nT$, each of which is separable in $P$   because $nT$ is separable.

 The general case follows because $T$ is  a free abelian group
 of finite rank, and therefore every subgroup   is the intersection of a collection of finite index subgroups (because the quotient is clearly residually finite).
 \end{proof}

\subsection{Proof of Theorem \ref{rfhyp}}
First suppose that $G$ is torsion-free.
The result is true if $\Out(G)$ is finite, or if $G$ is the fundamental group of a closed surface (in the orientable case this was proved by Grossman \cite{Grossman}, and in the non-orientable case by
Allenby, Kim and Tang \cite{A-K-T}).
Otherwise, by Theorem~5.3 of \cite{GD}, the group $\Out(G)$ is virtually  a    product $\Z^q\times M$,
with $M$    a quotient of a finite direct product $\Pi=\prod_{v  }PMCG^\partial(G_v)$; here $G_v$ is a surface group $H$ as in Subsection \ref{surf}
(a QH vertex group of the cyclic JSJ decomposition of $G$), and we
denote by $\calt _{ G_v}\lhd PMCG^\partial(G_v)$ the corresponding group of twists. Moreover, the kernel $Z$ of the map from $\Pi$ to $M$ is
contained in the free abelian group $\calt _\Pi=\prod_{v  }\calt _{ G_v}$.

Lemma \ref{omcg} implies that
$PMCG^\partial(G_v)/n\calt _{ G_v}$ is residually finite  for all sufficiently large $n$.   It follows that  $Z$ is separable  in $\Pi$ by Lemma \ref{alg}.
Thus $M$, and  therefore also $\Out(G)$, are residually finite.

Now suppose that $G$ is only virtually torsion-free, and let $N\lhd G$ be a torsion-free normal  subgroup of finite index.
If $G$ is virtually cyclic, $\Out(G)$ is finite (cf.\  \cite[Lemma 6.6]{M-O-norm_aut}). Otherwise,
$N$  has trivial center, so some finite index subgroup   of $\Out(G)$ is isomorphic to the quotient of a subgroup of $\Out(N)$ by  a finite normal subgroup  (see \cite[Lemma 5.4]{G-L-i} or Lemma \ref{lem:outnormsub} below).
 We have shown above that $\Out(N)$ is residually finite, and therefore  so is $\Out(G)$.

\begin{rem} An alternative method to prove Theorem \ref{rfhyp} could employ Funar's results about residual finiteness of central extensions of mapping class groups \cite{Funar}.
However, writing a complete proof using this approach would still require substantial work, for instance because the surfaces involved may be non-orientable.
\end{rem}

\section{Relatively hyperbolic groups and trees}

In this section we recall basic material about relatively hyperbolic groups and trees.

\subsection
{Relatively hyperbolic groups} \label{defrhyp}

There are many equivalent definitions of relatively hyperbolic groups in the literature. The definition we give below is due to B.\  Bowditch \cite{Bow}; for its equivalence to the other definitions see  \cite{Hruska} or \cite{Osin-RHG}.
  In this paper we will always assume that $G$ and all the groups $P_i\in\calp$ are finitely generated. (Note that if $G$ is hyperbolic relative to a finite family of finitely generated subgroups
then $G$ is itself finitely generated. This follows, for example, from the equivalence of Definition \ref{def:rh} with Osin's definition \cite[Def. 1.6]{Osin-RHG}.)

\begin{df}[Def.~2 of \cite{Bow}]  \label{def:rh}
Consider a group $G$ with a family of subgroups $\calp=\{P_1,\dots,P_k\}$. We will say that $G$ is hyperbolic relative to $\calp$ if $G$ admits a simplicial action on a connected graph $\mathcal{K}$
such that:
\begin{itemize}
  \item $\mathcal{K}$ is $\delta$-hyperbolic for some $\delta\ge 0$, and for each $n \in \N$ every edge of $\mathcal{K}$ is contained in finitely many simple circuits of length $n$;
  \item the edge stabilizers for this action of $G$ on $\mathcal{K}$ are finite, and there are finitely many orbits of edges;
  \item $\calp$ is a set of representatives of conjugacy classes of  the infinite vertex stabilizers.
\end{itemize}
\end{df}

We usually assume that   each $P_i$ is a \emph{proper} subgroup of $G$, i.e., $P_i\ne G$  (as any $G$ is hyperbolic relative to itself).

A subgroup $H\leqslant G$ is \emph{elementary} if it is virtually cyclic (possibly finite) or \emph{parabolic}
(contained in  a conjugate of some $P_i$). Any infinite elementary subgroup $H$  is contained in a unique maximal elementary subgroup $\hat H$, and $Z_G(H)\inc \hat H$.
We say that $G$ itself is  \emph{elementary} if it is virtually cyclic or equal  to some $P_i$.

It is well known that, if some $P_i$ is virtually cyclic (or,
more generally, hyperbolic), then $G$ is hyperbolic relative to the family $\calp\setminus\{P_i\}$ -- see, for example, \cite[Thm.\  2.40]{Osin-RHG}.
In the context of Corollary \ref{main1}, we may therefore assume that no $P_i$ is virtually cyclic. We do not wish to do so in Theorem \ref{gener}, because  it may happen that
$G$ is one-ended relative to $\calp$, but not relative to $\calp\setminus\{P_i\}$ with $P_i$ infinite and virtually cyclic.
For simplicity, however, we assume in most of the paper that no $P_i$ is virtually cyclic. The (few) changes necessary to handle the general case are explained in Subsection \ref{vcpi}.

\subsection{The canonical splitting of a one-ended relatively hyperbolic group}\label{JSJ}

Assume  that $G$ is one-ended, or, more generally, one-ended relative to $\calp$: it does not split over a finite group relative to $\calp$ (i.e., with every $P_i$ fixing a point in the Bass-Serre tree). Then there is a canonical
JSJ tree $T$  over elementary subgroups relative to $\calp$ (see \cite[Corollary 13.2]{GL3b} and \cite{GL6}).
Canonical means, in particular, that $T$ is invariant under the natural action of $\Out(G;\calp)$.

The tree $T$ is equipped with an action of $G$.
We denote by $G_v$ the stabilizer of a vertex $v$, by $G_e$ the stabilizer of an edge $e$ (it is elementary).
They are infinite and finitely generated \cite{GL6}. If $e=vw$, we say that $G_e$ is an incident edge stabilizer in $G_v$ and $G_w$.

We   also consider the quotient graph of groups $\Gamma=T/G$. We then denote by $G_v$ the group carried by a vertex $v$, by $G_e$ the group carried by an edge $e$. If $v$ is an  endpoint of $e$, we often identify $G_e$ with a  subgroup of $G_v$, and we say that $G_e$ is an incident edge group at $v$.

Being a tree of cylinders (see \cite {GL3b}), $T$  is bipartite, with vertex set $A_0\cup A_1$.
The stabilizer of a vertex $v_1\in A_1$ is a maximal elementary subgroup (we also say that $v_1$ is an elementary vertex). The stabilizer of an edge $\varepsilon=v_0v_1$ (with  $v_i\in A_i$)
is a maximal elementary subgroup of $G_{v_0}$  (i.e., it is maximal among elementary subgroups contained in $G_{v_0}$), but $G_\varepsilon$ is not necessarily maximal elementary in $G_{v_1}$ or in $G$.

Vertices  in  $A_0$ have non-elementary stabilizers.
A vertex  $v\in  A_0$ (or its stabilizer $G_v$) is either  rigid or QH (quadratically hanging).
A rigid $G_v$ does not split over an elementary subgroup relative to parabolic subgroups and incident edge stabilizers. Through the Bestvina-Paulin method and Rips theory, this has strong implications on its
automorphisms (see \cite{GL6}). This will be the key point in the proof of Lemma \ref{reduc}.

To describe QH vertices,
it is more convenient to consider a QH vertex group $G_v$ of the graph of groups $\Gamma$. First suppose that $G$ is torsion-free. Then $G_v$ may be identified with
the fundamental group of a (possibly non-orientable) compact hyperbolic surface $\Sigma_v$ whose boundary is non-empty (unless $G_v=G$).

Moreover, each incident edge group $G_e$ is (up to conjugacy) the fundamental group $H_\calc$
of a boundary component $\calc$ of $\Sigma_v$. Different edges correspond to different boundary components.
Conversely, if no $P_i$ is cyclic, the fundamental   group $H_\calc$ of every boundary component $\calc$  is an incident edge group.
If cyclic $P_i$'s are allowed, it may happen that $H_\calc$ is not an incident edge group; it is then conjugate to some $P_i$.

If torsion is   allowed, we only have an
exact sequence   $$\{1\} \to F \to G_v \stackrel{\xi}{\to} P \to \{1\},$$
where $F$ is a finite group and $P$ is the fundamental group of a compact
hyperbolic 2-orbifold $\mathcal{O}_v$. If $\calc$ is a boundary component of $\mathcal{O}_v$, its fundamental group $\pi_1(\calc)\inc P$ is
infinite cyclic or infinite dihedral; one defines $H_\calc\inc G_v$ as its  full preimage under  $\xi$.  It is an incident edge group or is conjugate to a virtually cyclic $P_i$.

Note that, in all cases, a QH vertex  stabilizer $G_v$ of $T$ is virtually free (unless $G_v= G$), and  stabilizers of incident edges are
virtually cyclic.

The tree $T$ is relative to $\calp$: every $P_i$ fixes a point. If $P_i$ is not virtually cyclic, it equals  the stabilizer of a vertex $v_1\in A_1$ or is contained in some $G_{v_0}$ with $v_0$ rigid
(it may happen that $P_i=G_{v_1}\varsubsetneq G_{v_0}$). In particular, the intersection of $P_i$ with a QH vertex group is virtually cyclic.  If $P_i$ is virtually cyclic and infinite,
there is the additional possibility that it is contained in a QH vertex stabilizer (and conjugate to an  $H_\calc$ as above).

  \subsection{The automorphism group of a tree} \label{autt}
Let $T$ be any tree with an action of a finitely generated group $G$. We assume that the action is minimal (there is no proper $G$-invariant subtree), and $T$ is not a point or a line.
Let $\Out(G;T)\inc\Out(G)$ consist of  outer automorphisms $\Phi=\hat\alpha$ leaving $T$ invariant:  in other words, $\Phi$ comes from an automorphism $\alpha\in\Aut(G)$ such that
there is an isomorphism $H_\alpha:T\to T$ satisfying  $\alpha(g)H_\alpha= H_\alpha g$ for all $g\in G$.
We study $\Out(G;T)$ as in Sections 2-4 of \cite{GD}.

It is more convenient to consider the  quotient graph of groups $\Gamma=T/G$. It is finite, and the maps $H_\alpha$ induce an action of
$\Out(G;T)$ on $\Gamma$.  We denote by $\Out_0(G;T)\leqslant\Out(G;T)$ the finite index subgroup consisting of automorphisms acting trivially on  $\Gamma$.

We denote  by $V$ the vertex set of $\Gamma$, by $E$ the set of oriented edges, by $E_v$ the set of oriented edges $e$ with origin $o(e)=v$ (incident edges at $v$),
by $\cale$ the set of non-oriented edges. We write $G_v$ or $G_e$ for the group attached to a vertex or an edge, and we view $G_e$ as a subgroup of $G_v$ if $e\in E_v$ (incident edge group).

 For $v\in V$, we   define groups $PMCG(G_v)\leqslant \Out(G_v)$  and $PMCG^\partial(G_v)$  as in
  Subsection \ref{relaut},
using as  $C_i$'s
 the incident edge groups  ($s$ is the valence of $v$ in $\Gamma$, and there are repetitions  if $G_e=G_{e'}$ with $e\ne e'$).
We denote by $\pi_v:PMCG^\partial(G_v)\to
PMCG(G_v)$ the natural projection.

There is a natural   map (extension by the identity)   $\lambda_v:PMCG^\partial(G_v)\to \Out_0(G;T)$ (see \cite[Section 2.3]{GD}).
For instance, if $\Gamma$ is an amalgam $G=G_v*_{G_e}G_w$, and  $\psi\in PMCG^\partial(G_v)$ is represented by $(\alpha;a_1)$, the image of $\psi$
is represented by the automorphism of $G$ acting as $\alpha$ on $G_v$ and as conjugation by $a_1$ on $G_w$.
Elements in the image of $\lambda_v$ act as  inner automorphisms of $G$    on   $G_w$ for $w\ne v$.
The maps $\lambda_v$ have commuting images and  fit together in a map   $$\lambda:\prod _{v\in V }PMCG^\partial(G_v)\to \Out_0(G;T).$$

There is also a map $$\rho=\prod_{v\in V} \rho_v:\Out_0(G;T)\to\prod _{v\in V }\Out(G_v)$$ recording the action of automorphisms on vertex groups, and
the projection $$\pi=\prod\pi_v:\prod _{v\in V }PMCG^\partial(G_v)\to \prod _{v\in V }PMCG (G_v)\leqslant \prod _{v\in V }\Out(G_v)$$ factors as $\pi=\rho \circ\lambda$.

We let $\Out_1(G;T)\inc\Out_0(G;T)$
be the image of $\lambda$, and
  $$\rho_1:
 \Out_1(G;T)\to \prod _{v\in V }PMCG (G_v) $$ the restriction of  $\rho$.  In general $\Out_1(G;T)$ is smaller than $\Out_0(G;T)$ because elements of
$\Out_1(G;T)$ are required to map into $PMCG (G_v)$ for all $v$, and also since $\ker\rho$ may fail to be contained in $\Out_1(G;T)$ because of ``bitwists''
(which  will not concern us here, see the proof of Lemma \ref{reduc}).

To sum up, we have written $\pi$ as the product of two epimorphisms
$$\prod _{v\in V }PMCG^\partial(G_v)\stackrel{\lambda}{\twoheadrightarrow}
 \Out_1(G;T)\stackrel{\rho_1}{\twoheadrightarrow}   \prod _{v\in V }PMCG (G_v) .$$

We now study the \emph{group of twists} $\calt=\ker\rho_1=\ker\rho\cap\Out_1(G;T)$. It is generated by the commuting subgroups $\lambda_v(\calt_v)$, where $\calt_v$
 is the kernel of the projection $\pi_v:PMCG^\partial(G_v)\to  PMCG (G_v)$.
    As in Subsection \ref{relaut}, $\calt_v$ is the quotient of  $\prod _{e\in E_v}Z_{G_v}(G_e)$   by
   $Z(G_v)$ (embedded diagonally), which we call a \emph{vertex relation}. The image of an element $z\in Z_{G_v}(G_e)$ in $\Out(G;T)$ is
   the \emph{twist by $z$ around  $e$ near $v$} (note that $z$ does not have to belong to $G_e$).
  For instance, in the case of the amalgam considered above, it acts as the identity on $G_v$ and as conjugation by $z$ on $G_w$.

The group $\calt$ is generated by the product $\prod _{e\in E}Z_{G_{o(e)}}(G_e)$.
A complete set of relations is given by the vertex relations $Z(G_v)$ (with $Z(G_v)$ embedded diagonally into the factors $\prod Z_{G_{o(e)}}(G_e)$
such that $o(e)=v$) and the \emph{edge relations} $Z(G_e)$ (with $Z(G_e)$ embedded diagonally into the  factors $Z_{G_v}(G_e)$ and $Z_{G_w}(G_{\bar e})$ if
$e$ is an oriented edge $vw$ and $\bar e=wv$). In the case of an amalgam $G=G_v*_{G_e}G_w$, the edge relation simply says that conjugating both $G_v$ and
$G_w$ by $z\in Z(G_e)$ defines an inner automorphism of $G$.

In other words, $\calt $ is the quotient of $$\prod _{e\in E}Z_{G_{o(e)}}(G_e)$$ by the image of $$\prod _{v\in V}Z({G_v})\times\prod _{\varepsilon\in\cale}Z(G_\varepsilon),$$
where the products are taken over all oriented edges,
all vertices, and all non-oriented edges respectively.

Dividing $\prod _{e\in E}Z_{G_{o(e)}}
(G_e)$ by the vertex relations yields $\prod _{v\in V }PMCG^\partial(G_v)$. The edge relations generate the kernel of
$\lambda:\prod _{v\in V }PMCG^\partial(G_v)\to \Out_0(G;T)$. Note that    $\ker\lambda\inc\ker\pi=\prod_{v\in V}\calt_v$.

  \begin{example*}
  In the case of an amalgam, $\calt$ is the image of the map $p:Z_{G_v}(G_e)\times Z_{G_w}(G_e)\to\Out(G)$ sending $(a,b)$ to the class of the automorphism acting on $G_v$ as conjugation by $b$ and on $G_w$ as conjugation by $a$. The kernel of $p$ is generated by the elements $(a,1)$ with $a\in Z(G_v)$ and $(1,b)$ with $b\in Z(G_w)$ (vertex relations), together with the elements $(c,c)$ with $c\in Z(G_e)$ (edge relations). We have $\calt_v=Z_{G_v}(G_e)/Z(G_v)$ and $\calt_w=Z_{G_w}(G_e)/Z(G_w)$. The kernels of $\lambda:PMCG^\partial(G_v)\times PMCG^\partial(G_w)$, and of its restriction to $\calt_v\times \calt_w$, are generated by $Z(G_e)$.
\end{example*}

The following lemma will be used in Section \ref{tfrh}.
\begin{lem} \label {disj}
Let $W\inc V$. Assume
  that $G_v$ has trivial center if $v\notin W$, and that  $\Gamma$ has no  edge with both endpoints in $W$. Then the map $\lambda_W:\prod_{v\in W }PMCG^\partial(G_v)\to \Out(G;T)$ is injective.
\end{lem}

\begin{proof}
The kernel of  $\lambda_W$ is the intersection of $\ker\lambda$ with $\prod_{v\in W }\calt_v\inc \prod_{v\in V }\calt_v$.
If $v\notin W$, the group $G_v$ has trivial center, so the product  $\prod _{e\in E_v}Z
(G_e)$, taken over all edges with origin $v$, injects into  $\calt_v$. This implies that, if a product $z=\prod_{\varepsilon\in\cale}z_\varepsilon$ with  $z_\varepsilon\in Z(G_\varepsilon)$
maps to $\prod_{v\in W }\calt_v$,  it cannot involve   edges having an endpoint outside of $W$. Thus $z$ is trivial since all edges are assumed to have an endpoint not in $W$.
\end{proof}

\begin{rem} If edges with both endpoints in $W$ are allowed,  the proof shows that the kernel of $\lambda_W$ is generated by the edge relations associated to   these edges.
\end{rem}

\section{Getting rid of the rigids}\label{sec:rid}

Let $G$ be a group
hyperbolic relative to  a family of  finitely generated  subgroups $\calp=\{P_1,\dots,P_k\}$, and   one-ended relative to $\calp$.
  In this section we assume that no $P_i$ is virtually cyclic
(see Subsection \ref{vcpi} for a generalization).

We consider the canonical elementary JSJ tree $T$ relative to $\calp$ (see Subsection \ref{JSJ}). It is invariant under  $\Out(G;\calp)$, so $\Out(G;\calp)\leqslant \Out(G;T)$, and
bipartite: each edge joins a vertex with elementary stabilizer to a  vertex with non-elementary (rigid or QH) stabilizer. In particular, $T$
cannot be a line; we assume that it is not a point.

As above, we consider the quotient graph of groups $\Gamma=T/G$,  with vertex set $V$. Just like those of  $T$, vertices of $\Gamma$ (and their groups) may
  be elementary, rigid, or QH.  We partition $V$ as $V_E\cup V_R\cup V_{QH}$ accordingly; each edge has exactly one endpoint in $V_E$.

\begin{lem}\label{tpre}
$\calt$ is generated by the groups $\lambda_w(\calt_w)$ with $w\in V_E$: twists near vertices in $V_E$ generate the whole group of twists of $T$.
\end{lem}

\begin{proof}
If  $e=vw$ is any edge with $v\in V_{QH}\cup V_R$ (and therefore $w\in V_E$), then $G_e$ is a maximal elementary subgroup of $G_v$, so $Z_{G_v}
(G_e)=Z(G_e)$  since $Z_{G_v}(G_e)\inc G_e$. We can then   use edge relations to view twists around $e$ near $v $ as twists near $w$.
\end{proof}

\begin{lem}\label{reduc}  Let $\Out^r(G)$ be the image of the restriction $$\lambda_{E, QH}:\prod _{v\in V_E\cup V_{QH}}PMCG^\partial(G_v)\to \Out(G)
 $$
 of $\lambda$
  (see Subsection \ref{autt}).
 Then $\Out^r(G)$  is contained in $\Out(G;\calp)$ with
 finite index.
\end{lem}

\begin{proof}
We first prove $\Out^r(G)\inc \Out(G;\calp)$.
 Recall that $P_i$ is assumed not to be virtually cyclic, so (up to conjugacy) it  is  equal to an elementary vertex group or is contained in a rigid vertex group.   In either case elements of $\Out^r(G)$ map $P_i$ to a conjugate (trivially if $P_i$ is contained in a rigid group).
In fact, $\Out^r(G)$ is contained in  $\Out_0(G;\calp)=\Out(G;\calp)\cap\Out_0(G;T)$, a finite index subgroup of  $\Out(G;\calp)$. We   show that $\Out^r(G)$ has finite index  in  $\Out_0(G;\calp)$.

Recall the maps
$\rho_v:\Out_0(G;T)\to \Out(G_v)$, and consider their restrictions to the subgroup $\Out_0(G;\calp)$. It is shown in  Proposition 4.1 of
\cite{GL6}  that the image of such a restriction  is finite if $v\in V_R$ (this is a key property of rigid vertices),
 contains $PMCG(G_v)$ with finite index   if $v\notin V_R$  (note that $PMCG(G_v)$ is   $\Out(G_v;\Inc_v^{(t)})$ in \cite{GL6};
 the assumption that no $P_i$ is virtually cyclic is used to ensure $\calb_v=\Inc_v$).

 Now consider the homomorphism $\rho :\Out_0(G;T)\to\prod _{v\in V }\Out(G_v)$. The image of $\Out^r(G)$
 is $\prod _{v\in V_E\cup V_{QH}}PMCG (G_v)$, so it has finite index in
the image of $\Out_0(G;\calp)$.
We complete the proof by showing that $\Out^r(G)$ contains $\ker\rho$.

It is pointed out in   Subsection 3.3 of \cite{GL6}  that $\ker\rho$ is equal to the group of twists $\calt$
(because,  if $v\notin V_E$, then incident edge groups are equal to their normalizer in $G_v$).
By Lemma \ref{tpre}, $\calt$ is generated by twists near vertices in $V_E$. These belong to the image of $\lambda_{E,QH}$, so $\calt\inc\Out^r(G)$.
\end{proof}

We shall now change the graph of groups $\Gamma$ into a new graph of groups $\Gamma'$. We do not change the underlying graph, or edge groups, or vertex groups $G_v$ for $v\in V_E\cup V_{QH}$, but for $v\in V_R$ we   replace $G_v$ by a   group $G'_v$ defined as follows.

If $e\in E_v$ is an incident edge, $G_e$ is a maximal elementary subgroup of $G_v$, in particular it contains the finite group $Z(G_v)$.
Consider   the groups $G_e$, for $e\in E_v$, as well as $\Z\times Z(G_v)$. All these groups contain $Z(G_v)$, and we define $G'_v$ as
their free amalgam over $Z(G_v)$,  i.e., $G'_v$ is obtained from the free product $ (*_{e\in E_v}G_e)*  \left(\Z\times Z(G_v) \right)$ by identifying all copies of $Z(G_v)$.

The  inclusion   from $G_e$ into the new vertex group $G'_v$ is the  obvious one. Note that
  $G'_v$
   is not one-ended relative to incident edge groups (because of the factor $\Z\times Z(G_v)$), and is residually finite if the $P_i$'s are (as an amalgam of residually finite groups over a finite subgroup).
We denote by $G'$ the fundamental group of $\Gamma'$.

\begin{lem} \label{rid}
The finite index subgroup $\Out^r(G)\inc \Out(G;\calp)$ is isomorphic to a subgroup of  $\Out(G')$.
\end{lem}

\begin{proof}
Since nothing changes near vertices in $V_E\cup V_{QH}$, we still have a map
$$\lambda'_{E,{QH}}:\prod _{v\in V_E\cup V_{QH}}PMCG^\partial(G_v)\to \Out(G').$$   It suffices to show $\ker\lambda'_{E,{QH}}=\ker\lambda_{E,{QH}}$.  Recall that the kernel of $\lambda_{E,{QH}}$
is the same as the kernel of its restriction  to $\prod _{v\in V_E\cup V_{QH}}\calt_v$, and similarly for $\lambda'_{E,{QH}}$.
For $v\in V_R$, the group $G'_v$ was defined so that the groups
  $Z(G_v)$ and $Z_{G_v}
(G_e)$ do not change, so $\calt_v$ does not change, and neither does the map  from $\prod_{v\in V}\calt_v$ to $\Out(G)$.  The result follows.
\end{proof}

\section{Torsion-free relatively hyperbolic groups} \label{tfrh}

This section is devoted to a proof of  Theorem \ref{gener} in the torsion-free case.

\begin{thm} \label{thm:t-f}
 Let $G$ be a torsion-free group hyperbolic relative to  a family of   proper   finitely generated residually finite subgroups $\calp=\{P_1,\dots,P_k\}$, none of which is   cyclic.   If $G$ is one-ended relative to $\calp$, then $\Out(G; \calp)$ is residually finite.
 \end{thm}

See Remark \ref{virtcyc} for the case   when some of the $P_i$'s are allowed to be cyclic.

\begin{proof}

As in the previous section, let $\Gamma$ be the canonical JSJ decomposition   of $G$. First suppose that
 $\Gamma$ is trivial (a single vertex $v$). If $v$ is rigid, then $\Out(G;\calp)$ is finite (see \cite{GL6}). If $v$ is QH, then $G$ is a closed surface group and $\Out(G)$ is residually finite
by \cite{Grossman,A-K-T}.   The   case $v\in V_E$ cannot occur since $P_i\ne G$. From now on, we suppose that $\Gamma$ is nontrivial.

By Lemma \ref{reduc}, it is enough to show that $\Out^r(G)$ is residually finite.
Given a nontrivial $\Phi\in\Out^r(G)$, we want to map $\Out^r(G)$ to a     finite group without killing $\Phi$.
Note that it is enough to map  $\Out^r(G)$ to a   residually finite group without killing $\Phi$.

Using Lemma \ref{rid}, we view $\Out^r(G)$ as a group of  automorphisms of $G'$. To simplify notation, we will not write the superscripts $'$ unless necessary.

Recall the epimorphisms
$$
  \prod _{v\in V_E\cup V_{QH}}PMCG^\partial(G_v) \to
 \Out^r(G)\to\prod _{v\in V_E\cup V_{QH}}PMCG (G_v)
 $$
 induced by  $\lambda_{E,{QH}}$ and $\prod_{v\in V_E\cup V_{QH}}\rho_v$.
Write $\Phi$ as the image, under $\lambda_{E,{QH}}$, of a tuple $(\Phi_v) \in   \prod _{v\in V_E\cup V_{QH}}PMCG^\partial(G_v)$.
 If $v\in V_{QH}$,
the group  $PMCG(G_v)$ is residually finite since $PMCG(G_v)\inc\Out(G_v)$ with $G_v$   free of finite rank, so   we may assume that $\Phi$ maps trivially to $PMCG(G_v)$  for $v\in V_{QH}$.
This means that $\Phi_v\in \calt_v\inc PMCG^\partial(G_v)$ for $v\in V_{QH}$.
By Lemma \ref{tpre}, we may use edge relations
to find a   representative of $\Phi$ with $\Phi_v=1$ for $v\in V_{QH}$.  We fix such a representative, and we choose
  $u\in V_E$ with $\Phi_{u}\ne1$.

 Next we   fix a characteristic finite index subgroup $N_v\lhd G_v$
for each $v\in V_E$,   and we denote $H_v=G_v/N_v$. Since $G_v$ is assumed to be residually finite (it is cyclic or conjugate to a $P_i$), we may also require that, if $e=vw$ is an edge with $w\in V_{QH}$,
then the image of $G_e\cong\Z$
in  $H_v $ has order $n_e\ge 4$.  As in the   proof of Lemma  \ref{rfd},
we also require that $\Phi_{u}$ maps to a non-trivial element under the natural homomorphism from
$PMCG^\partial(G_{u})$ to $PMCG^\partial(H_{u})$.

We now construct a new graph of groups $\Gamma''$, with the same underlying graph as $\Gamma$ and $\Gamma'$, with vertex groups $G''_v=H_v$ at vertices $v\in V_E$.
We describe  the edge groups, and  the other vertex groups. The inclusions from edge groups to vertex groups  will be  the obvious ones.

Given an edge  $e=vw$ with $v\in V_E$, its group $G''_e$ in $\Gamma''$ is the image of $G_e$ in $H_v$, a cyclic group of order $n_e\ge4$.

 If $v\in V_R$, its group in $\Gamma'$ is the free product $ (* G_e)*\Z$, with the first product taken over all $e\in E_v$.
We define its new group as
$G''_v=(* G''_e)*\Z$.

If $v\in V_{QH}$, its group in $\Gamma$ and $\Gamma'$ is
a surface group
$\pi_1(\Sigma_v)$, with boundary components of $\Sigma_v$ corresponding to incident edges (see Subsection \ref{JSJ}).
We define $G''_v$ as the fundamental group of the closed orbifold obtained from $\Sigma_v$ by replacing the boundary component associated to an edge $e$ by a conical point of order $n_e$.
The assumption $n_e\ge4$ ensures  that the orbifold is hyperbolic (see \cite{Th}, 13.3.6), so $G''_v$ is a non-elementary subgroup of   Isom$(\Hy^2)$.
In particular, $G''_v$ has trivial center
 and is residually finite.

The fundamental group $G''$ of $\Gamma''$ is a quotient of $G'$. It is residually finite, as the fundamental group of a graph of
groups with residually finite vertex groups and finite edge groups.
We show that the splitting $\Gamma''$ of $G''$ is nontrivial.

If $v\notin V_ E$, the group $G''_v$ is infinite (because of the factor $\Z$ in the rigid case), so triviality would imply that $u$ has valence 1 and the   incident edge group in $\Gamma''$ is equal to $G''_{u }$. This is impossible because $PMCG^\partial(G''_{u})$ is non-trivial (it contains the image of $\Phi_u$, which is assumed to be non-trivial).

Since $\Gamma''$ is a   splitting  over finite groups,
$G''$ has infinitely many ends.
By \cite[Thm.\  1.5]{M-O-norm_aut}, $\Out(G'')$ is residually finite. We conclude the proof by constructing a map from $\Out^r(G)$ to $\Out(G'')$ mapping $\Phi$ non-trivially.

Let $v\in V_E\cup V_{QH}$.
Since the kernel of the map $G_v\to G''_v$ is invariant under $PMCG(G_v)$, there is an induced map from $PMCG^\partial(G_v)$ to $PMCG^\partial(G''_v)$.
Similarly, the kernel of the projection from $G'$ to $G''$ is invariant under the image of $\Out^r(G)$ in $\Out(G')$ (see Lemma \ref{rid}),
so there is an induced map from $\Out^r(G)$ to $\Out(G'')$. These maps fit in a commutative diagram
 $$
 \xymatrix{
  \prod _{v\in V_E\cup V_{QH}}PMCG^\partial(G_v)
  \ar[r]\ar[d]&  \prod _{v\in V_E\cup V_{QH}}PMCG^\partial(G''_v) \ar[d]  \\
  \Out^r(G)\ar[r]& \Out(G'').
}
$$
Since $\Phi_v=1$ for $v\in V_{QH}$, and $\Phi_{u}$ maps non-trivially   to $PMCG^\partial(H_{u})$,  Lemma \ref{disj} (applied in $\Gamma''$ with $W=V_E$) implies that $\Phi$ maps  non-trivially to $\Out(G'')$.
\end{proof}

The main difficulty in extending this proof to groups with torsion lies in defining $G''_v$ for $v\in V_{QH}$ when $G_v$   contains a nontrivial finite normal subgroup $F$.  This may be done using small cancellation techniques, but we prefer to give a   different proof   (using Grossman's method and Dehn fillings) in the general case.

\begin{rem}\label{virtcyc} Theorem \ref{thm:t-f} holds if the $P_i$'s are allowed to be   cyclic
( recall that an infinitely-ended $G$ may   become one-ended relative to $\calp$ if   cyclic $P_i$'s are added).
There is a technical complication due to the fact that, if $G_v=\pi_1(\Sigma_v)$ is QH, there may exist boundary components of  $\Sigma_v$ whose fundamental group equals some $P_i$ (up to conjugacy)
 but is not an incident edge group. Because of this, one must  change the definition of $PMCG^\partial(G_v)$ and $PMCG (G_v)$ slightly. See Subsection \ref{vcpi} for details.
\end{rem}

\section{Groups with torsion}\label{gtor}
The goal of this section is to prove Theorem \ref{gener} (Subsections \ref{conc} and \ref {vcpi}) and   Corollary~\ref{main1} (Subsection \ref{pfcor}).

\subsection{Finitary fillings of relatively hyperbolic groups}

\begin{df}\label{def:fin_fil} Let $\mathcal{I}$ be a property of groups, let $G$ be a group, and let  $\{H_\lambda\}_{\lambda \in \Lambda}$ be a non-empty collection of subgroups of $G$. We will say that
\emph{most finitary fillings of $G$ with respect to $\{H_\lambda\}_{\lambda \in \Lambda}$ have property $\mathcal{I}$} provided
there is a finite subset
$S \subset G\setminus\{1\}$ such that, for any family of finite index normal subgroups $N_\lambda \lhd H_\lambda$
with $N_\lambda \cap S=\emptyset$ for all $\lambda\in\Lambda$,   the following two conditions hold:
\begin{itemize}
\item
for each $\lambda\in\Lambda$ one has $N \cap H_\lambda=N_\lambda$, where $N:=\llangle N_\lambda \mid \lambda \in \Lambda\rrangle^G$;
\item the quotient $G/N$  satisfies $\mathcal I$.
\end{itemize}
Any quotient $G/N$ as above will be called a \emph{finitary filling of} $G$ with respect to $\{H_\lambda\}_{\lambda \in \Lambda}$. The finite subset $S$ will be called the  \emph{obstacle subset}.
\end{df}

In this work we will mainly be concerned with the case when $\mathcal I$ is the property of being residually finite, or conjugacy separable.

\begin{rem}\label{rem:gr_gps_w_f_edges-ceg}
If $G=A*_C B$, where $C$ is finite and $A,B$ are residually finite, then most finitary fillings of $G$ with respect to any of the families
$\{A \}$, $\{B\}$ or $\{A,B\}$ are residually finite (it suffices to take $S:=C\setminus\{1\}$ and use the universal property of amalgamated free products;
$G/N$ will be an amalgam of residually
finite groups over $C$).
More generally, if $G$ is the fundamental group of a finite graph of groups with finite edge groups and residually finite vertex groups, then most
finitary fillings of $G$ with respect to any
collection of vertex groups will be residually finite.
\end{rem}

In this subsection, and the next, we will consider a graph of groups $  \Gamma$, with fundamental group $G$, satisfying the following assumption
(this will be applied to the graph $\Gamma'$ constructed in Section \ref{sec:rid}).

\begin{as}\label{df:gr-gr}
$\Gamma$ is a   connected finite bipartite graph with vertex set $V=V_1 \sqcup V_2$ (so   every vertex from $V_1$ is only adjacent to vertices from $V_2$ and vice-versa),
and $\Gamma$ is not a point.
Moreover, the following properties hold:

\begin{enumerate}
\item if $u \in V_1$ then the group $\G_u$ is residually finite;

\item if $v \in V_2$ then most finitary fillings of $\G_v$ with respect to $\{\G_{e_1},\dots,\G_{e_s}\}$ are residually finite,
with   $e_1,\dots,e_s$   the collection of all oriented edges of $\Gamma$ starting at $v$, and   $\G_{e_1},\dots,\G_{e_s}$   the corresponding edge groups;
\item   for every $u \in V_1$,  the group $G_u$ is a proper finitely generated subgroup of $G=\pi_1(\Gamma)$, and $G$ is hyperbolic relative to the family
$\{\G_u \mid u \in V_1\}$; in particular, $G$ is finitely generated.
\end{enumerate}
\end{as}

The main technical tool for our  approach is the theory of Dehn fillings in relatively hyperbolic groups, developed by Osin -- see \cite[Thm.\ 1.1]{Osin-CEP}
(in the torsion-free case this was also done independently  by  Groves and Manning -- see \cite[Thm.\ 7.2]{G-M}):

\begin{thm}\label{thm:periph_fil} Suppose that a group $G$ is hyperbolic relative to a family
of subgroups $\{H_\lambda\}_{\lambda \in \Lambda}$.
Then there exists a finite subset $S \subset G \setminus \{1\}$
with the following property. If  $\{N_\lambda\}_{\lambda \in \Lambda}$ is any collection
of subgroups such that $N_\lambda \lhd H_\lambda$ and $N_\lambda \cap S=\emptyset$ for all $\lambda \in \Lambda$, then: \\
1) for each $\lambda \in \Lambda$ one has  $H_\lambda \cap N=N_\lambda$, where $N:=  \llangle N_\lambda \mid \lambda \in \Lambda \rrangle^G$;
\\ 2) the quotient group $G/N $ is hyperbolic relative to the collection $\{H_\lambda/N_\lambda\}_{\lambda \in \Lambda}$.

Moreover,  for any finite subset $M \subset G$, there exists a finite subset $S(M) \subset G\setminus \{1\} $,
such that the restriction of the natural homomorphism
$G\to G/N$ to $M$ is injective whenever $N_\lambda \cap S(M)=\emptyset$ for all $\lambda \in \Lambda$.
  \qed
\end{thm}

Suppose that  $\Delta$ is any graph of groups with fundamental group $G$, and we are given normal subgroups $N_v\lhd G_v$ for each vertex $v$. Assume furthermore that
$N_v\cap G_e=N_w\cap G_e$ whenever $e=vw$ is an edge of $\Delta$  (as usual, we view $G_e$ as a subgroup of both $G_v$ and $G_w$;
to be precise, we want   the preimages of $N_v$ and $N_w$, under the embeddings of $G_e$ into $G_v$ and $G_w$ respectively, to coincide).

Then we can construct a ``quotient graph of groups'' $\bar\Delta$ as follows: the underlying graph is   the same as in $\Delta$, the vertex group at a vertex $v$ is $G_v/N_v$, the group carried by
$e=vw$ is $G_e/(N_v\cap G_e)$,
and the inclusions are the obvious ones.
The fundamental group of $\bar \Delta$ is isomorphic to $G/ \llangle\cup_v N_v \rrangle^G$.

Now suppose that $\Gamma$ is as in Assumption \ref{df:gr-gr}. For each $v\in V_2$, there is an obstacle set $S_v \subset \G_v\setminus \{1\}$, and we define $S:= \bigcup_{v \in V_2} S_{v}$.

\begin{lem}\label{qrf}
Consider an arbitrary family of  subgroups $\{N_u\}_{u \in V_1}$ such that $N_u \lhd \G_u$, $|\G_u:N_u|<\infty$ and $N_u \cap S=\emptyset$ for every $u \in V_1$.
 The group $\bar G:=G/ \llangle\cup_{u\in V_1} N_u \rrangle^G$ is the fundamental group of a quotient graph of groups $\bar \Gamma$ in which every $u\in V_1$
 carries $G_u/N_u$, and every $v\in V_2$ carries a residually finite group. In particular, $\bar G$ is residually finite.
\end{lem}

\begin{proof}
  For every edge
     $e=uv$ of $\Gamma$, with $u \in V_1$ and $v \in V_2$,
  we define a finite index normal subgroup $L_e \lhd \G_e$ by $L_e:=G_e \cap N_{u}$
(as above, we
view $G_e$ as a subgroup of both $G_u$ and $G_v$).
Now,  for each vertex $v \in V_2$ we let $M_v:= \llangle L_{e_1} \cup \dots \cup L_{e_s} \rrangle^{\G_v} \lhd \G_v$, where $e_1,\dots,e_s$ are the edges of $\Gamma$
starting at $v$. Observe that $L_{e_j} \cap S_v=\emptyset$ for $j=1,\dots,s$ by construction, hence $M_v \cap \G_{e_j}=L_{e_j}$ by Definition \ref{def:fin_fil}.

This shows that $\bar G$ is represented by a quotient graph of groups $\bar\Gamma$.  The group carried by
$u\in V_1$ is $G_u/N_u$, a finite group; in particular, edge groups are finite. The group carried by
 $v\in V_2$ is $\G_v/M_v$, which  is residually finite by Assumption \ref{df:gr-gr}. Thus $\bar G$ is residually finite.
\end{proof}

\begin{rem}\label{qrf+}
Since every $G_u$, with $u \in V_1$, is residually finite, a family of normal subgroups $\{N_u\}_{u \in V_1}$ as in  Lemma \ref{qrf} exists.
 We may even require $N_u\cap S'=\es$ if $S'$ is any given finite subset of $G \setminus \{1\}$.
\end{rem}

\begin{rem}  Lemma \ref{qrf} only requires the  second condition of Assumption \ref{df:gr-gr}.
\end{rem}

\begin{prop}\label{prop:G-rf} Let $G$ be the fundamental group of a graph of groups $\Gamma$ as in   Assumption~\ref{df:gr-gr}.
Then $G$ is residually finite.
\end{prop}

\begin{proof} Take any element $x \in G \setminus \{1\}$ and let $M:=\{1,x\} \subset G$.
Since each $\G_u$, $u \in V_1$, is residually finite, one can find sufficiently small finite index subgroups $N_u \lhd \G_u$ as in Lemma \ref{qrf}. We can also assume that each
$N_u$ is disjoint from the set $S(M)$ provided by Theorem \ref{thm:periph_fil}  (applied to $G$ relative to the $G_u$'s). If $N:=\llangle N_u \mid u \in V_1 \rrangle^G$, then $\bar G=G/N$ is residually finite by
Lemma \ref{qrf}, and the image of $x$ under the map $\kappa:G\to \bar G$  is non-trivial by the final claim of Theorem \ref{thm:periph_fil}. Composing $\kappa$
with a map from $\bar G$ to a finite group that does not kill $\kappa(x)$,
we get a finite quotient of $G$ in which the image of $x$ is non-trivial. This shows residual finiteness of $G$.
\end{proof}

  The above proposition is true even without the hypothesis that vertex groups from $V_1$ are finitely generated.

\subsection{Using Grossman's method}
 Let $G$ be
hyperbolic relative to  a family of proper finitely generated
 subgroups $\calp=\{P_1,\dots,P_k\}$.

Recall that an element $g \in G$ is called \textit{loxodromic}
if $g$ has infinite order and is not conjugate to an element of   $P_i$ for any $i$. Two elements $g,h\in G$ are said to be
\textit{commensurable}
in $G$ if there are $f \in G$ and $m,n \in \Z\setminus \{0\}$ such that $h^n=fg^mf^{-1}$
(we use the terminology of \cite{M-O-norm_aut}, where conjugate elements are considered commensurable).  An automorphism $\alpha \in {\rm Aut}(G)$ is called \textit{commensurating}
if $\alpha(g)$ is commensurable with $g$ for every $g \in G$.

It is known that
$G$ contains a unique maximal finite normal subgroup denoted $E(G)$ (see \cite[Cor.\  2.6]{M-O-norm_aut}); this subgroup contains the center of $G$
if $G$ is not virtually cyclic.

\begin{lem}  \label{lem:comm_aut_def}
Assume that $G$ is not virtually cyclic.
\begin{itemize}
\item[(1)]
If $E(G)=\{ 1\}$, then every commensurating automorphism of $G$ is inner.
\item[(2)] Suppose that $\alpha \in {\rm Aut}(G)$ is not a commensurating automorphism. Then
there exists a loxodromic element $g \in G$ such that $\alpha(g)$ is also loxodromic and $\alpha(g)$ is not commensurable with $g$ in $G$.
\end{itemize}
\end{lem}

\begin{proof} The statement (1) is proved in \cite[Cor.\ 1.4]{M-O-norm_aut}.

Suppose that $\alpha \in {\rm Aut}(G)$ is not commensurating. By \cite[Cor.\  5.3]{M-O-norm_aut} there exists
a loxodromic element $g_0 \in G$ such that $\alpha(g_0)$ is not commensurable with $g_0$ in $G$. The statement (2) now follows after applying \cite[Lemma 4.8]{M-O-norm_aut}.
\end{proof}

\begin{lem}{\rm (\cite[Lemma 7.1]{M-O-norm_aut})} \label{lem:non-comm}
Assume that $G$ is
hyperbolic relative to
$ \{P_1,\dots,P_k\}$, with $k\ge1$,
and $g, h \in G$ are two
non-commensurable loxodromic elements. Then $g$ and $h$ are loxodromic and
non-commensurable in most finitary fillings of $G$ with respect to
$ \{P_1,\dots,P_k\}$.
\qed
\end{lem}

  Recall that $G$ is \emph{conjugacy separable} if, given any non-conjugate elements $g,h \in G$, there exists a homomorphism $\varphi$ from $G$ to a
finite group such that $\varphi(g)$ and $\varphi(h)$ are not conjugate. Note that this is evidently stronger than residual finiteness of $G$.

\begin{prop}\label{prop:Out-rf}
Let $G$ be the fundamental group of a graph of groups $\Gamma$ as in   Assumption  \ref{df:gr-gr}.
  Suppose that $E(G)=\{1\}$, and
for each $v \in V_2$ most finitary fillings of
$\G_v$ with respect to $\{\G_{e_1},\dots,\G_{e_s}\}$
  (where $e_1,\dots,e_s$ is the list of edges of $\Gamma$ starting at $v$)
are \emph{conjugacy separable}. Then ${\rm Out}(G)$ is residually finite.
\end{prop}

The proof uses Grossman's method, which is based  on the following fact:

\begin{lem}\label{lem:cs->rf_out} Given a finitely generated group $G$ and $\alpha\in\Aut(G)$,
suppose that there is a homomorphism $\psi:G \to K$ with $K$ finite, and $g\in G$, such that $\psi(g)$ is not conjugate to $\psi(\alpha(g))$ in $K$. Then there is a
homomorphism $\hat\theta: {\rm Out}(G) \to L$ with $L$ finite such that
$\hat\theta(\hat\alpha) \neq 1$ in $L$, where $\hat \alpha$ is the image of $\alpha$ in $\Out(G)$.
\end{lem}

\begin{proof}
Since $G$ is finitely generated, there exists a characteristic finite index subgroup
$N \lhd G$ such that $N \le \ker \psi$. Let $\varphi: G \to G/N$ be the canonical epimorphism. Then  $\psi$ factors through
$\varphi$, hence $\varphi(g)$ is not conjugate to $\varphi(\alpha(g))$ in $G/N$. Observe that, as $N$ is characteristic in $G$, there are induced homomorphisms
$\theta: {\rm Aut}(G) \to {\rm Aut}(G/N)$ and $\hat\theta:{\rm Out}(G) \to L:={\rm Out}(G/N)$. Since $\varphi(g)$ is not conjugate to $\varphi(\alpha(g))$ in $G/N$, the automorphism $\theta(\alpha)$
is not inner and $\hat\theta(\hat\alpha)\neq 1$.
\end{proof}

\begin{proof}[Proof of Proposition \ref{prop:Out-rf}]

We may assume that   $G$ is not virtually cyclic, since   ${\Out}(G)$ is finite if it is   (see \cite[Lemma~6.6]{M-O-norm_aut}).
Applying Lemma \ref{qrf} and Remark~\ref{qrf+}, we find a quotient graph of groups $\bar\Gamma$ where vertices in $V_1$ carry finite groups and vertices in $V_2$ carry residually finite groups.
We write $\bar G=\pi_1(\bar \Gamma)$, and we let $\kappa:G\to \bar G$ be the projection.

We shall now enlarge  $S$ to ensure that $\bar \Gamma$ possesses additional properties.
First,  we may assume that vertices in $V_2$ carry a conjugacy separable group. By \cite{Dyer-2},
$\bar G $ is then conjugacy separable, as  the fundamental group of a finite graph of groups with conjugacy separable vertex groups and finite edge groups.

Now consider any  $\hat\alpha \in {\rm Out}(G)\setminus \{1\}$, represented by
$\alpha \in {\Aut}(G) \setminus \Inn(G)$.
Since $E(G)=\{1\}$, by Lemma \ref{lem:comm_aut_def} there exists a loxodromic element $g \in G$ such that $\alpha(g)$ is a loxodromic element not commensurable with $g$ in $G$.
By Lemma \ref{lem:non-comm} (applied to $G$ relative to the $G_u$'s), we   can enlarge the obstacle set $S$ to assume that the elements $\kappa(g)$ and $\kappa(\alpha(g))$  are non-commensurable in $\bar {G}$.

Since $\bar G$ is conjugacy separable, there exists a finite group $K$ and a homomorphism $\eta:\bar{G} \to K$
such that $\eta(\kappa(g))$ is not conjugate to $\eta(\kappa(\alpha(g)))$ in $K$. Thus, setting $\psi:=\eta \circ \kappa:G \to K$,
we can apply Lemma \ref{lem:cs->rf_out}
 to find a finite quotient of ${\rm Out}(G)$ separating
$\hat\alpha$ from the identity.
\end{proof}

\subsection{Quadratically hanging groups} \label{qhg}

In the next subsection, we will apply Propositions \ref{prop:G-rf} and  \ref{prop:Out-rf} to the canonical JSJ decomposition. In order to do this, we need to study finitary fillings of QH vertex groups.
We denote such a group by $O$.

Recall from Subsection \ref{JSJ}
that $O$ is an extension
$\{1\} \to F \to O \stackrel{\xi}{\to} P \to \{1\},$
where $F$ is a finite group and $P$ is the fundamental group of a
hyperbolic 2-orbifold $\mathcal{O}$ with boundary.
Consider full preimages  $C_i=\xi\m(B_i)$ of a set of representatives   $B_1,\dots,B_s$ of fundamental groups of components of the boundary of $\mathcal{O}$.

The goal of this subsection is the following statement:

\begin{prop} \label{lem:f-by-o_fil-cs}
 Most finitary fillings of $O$ with respect to $\calh=\{C_1,\dots,C_s\}$ are conjugacy separable.
\end{prop}

\begin{proof}
First assume   $F=\{1\}$, so $C_i=B_i$. In this case we shall see that most finitary fillings are fundamental groups of closed hyperbolic orbifolds.  Thus they are  virtually   surface groups,
hence conjugacy separable by a result of Martino \cite[Thm.\ 3.7]{Martino}.

We define an obstacle set $S=S_1\cup\dots\cup S_s$    in $P=\pi_1(\O)$ as follows. Let $r$ be a large integer (to be determined later).
Recall (see \cite{Th}) that each $B_i$ is either infinite cyclic or infinite dihedral.
In the cyclic  case, $B_i$ is generated by a single element $c$ of infinite order and we let $S_i:=\{c,c^2,\dots,c^r\}$. In the dihedral case,
$B_i$
 is generated by two involutions $a,b$ and we let $S_i:=\{a,b,ab,(ab)^2,\dots,(ab)^r\}$.

 Any
finite index normal subgroup   $K_i \lhd B_i $
with $K_i \cap S_i=\emptyset$
 is cyclic, generated by $c^{m_i}$
 or
$(ab)^{n_i}$ for some $m_i $  or $n_i$ larger than $ r$.
It follows that the quotient
$\pi_1(\O)/\llangle K_i \mid i=1,\dots,s \rrangle^{\pi_1(\O)}  $
is the fundamental group of a closed orbifold $\O'$, which is obtained from $\O$ by replacing each boundary component
by a conical point
(elliptic point, in the terminology of \cite{Th}) of order $m_i$ in the cyclic case,  and by a dihedral point
(corner reflector, in the terminology of \cite{Th}) of order $n_i$ in the dihedral case.

We claim that $\O'$ is hyperbolic if $r$ is large enough. By \cite[Thm.\ 13.3.6]{Th}, a  2-orbifold  admits a hyperbolic structure if and only if its Euler characteristic   is negative.
The Euler characteristic $\chi(\O')$
can be computed by the following formula (cf. \cite[13.3.3]{Th}, \cite{Scott}):
$$\chi(\O')=\chi(\O)+\sum
\frac{1}
{m_i} +\sum
\frac{1}{2n_i}.$$
Since $\chi(\O)$ is negative, so is
$\chi(\O')$ for $r$ large, and    the claim follows.   Defining $S$ using such an $r$, we deduce that most finitary fillings of $O$ are fundamental groups of closed hyperbolic orbifolds. This proves
  the proposition when   $F=\{1\}$.

In the general case, we have to use Dehn fillings (Theorem \ref{thm:periph_fil}). It is a standard fact \cite[Thm.\   7.11]{Bow} that $O$ is hyperbolic relative to the family $\calh=\{C_1,\dots, C_s\}$
(these are non-conjugate maximal virtually cyclic subgroups of the hyperbolic group $O$). Consider the obstacle set $\bar S=\xi\m(S)\cup S'$, where $S$ is the set constructed above in $\pi_1(\O)$
and $S'$ is provided by Theorem \ref{thm:periph_fil} (applied to $O$ and $\calh$).

Consider any collection of finite index normal subgroups $N_i \lhd C_i$ such that
$N_i \cap \bar S=\emptyset$, $i=1,\dots,s$, and set $N:=\llangle N_i \mid i=1,\dots,s\rrangle^O$. By Theorem \ref{thm:periph_fil} we have $N \cap C_i=N_i$ for each $i$,
so it remains to check that the quotient $O'=O/N$ is conjugacy separable.

Let $\varphi:O \to O'$ denote the natural epimorphism. Then $O'$ maps with finite kernel
$\varphi(F)$ onto $P/\xi(N) \cong P/\llangle \xi(N_i) \mid i =1,\dots,s \rrangle^P.$
Our choice of $\bar S$ guarantees that $ \xi(N_i)$ does not meet the set $S$, so $P/\xi(N)$ is the fundamental group of a hyperbolic orbifold $\O'$.
The exact sequence $\{1\} \to \varphi(F) \to O' \to \pi_1(\O') \to \{1\}$ implies  that $O'$ is virtually a surface group (see \cite[Thm.\ 4.3]{Martino}), hence it is conjugacy separable as above.
\end{proof}

\subsection{Conclusion}\label{conc}
We prove Theorem \ref{gener}, starting with a couple of lemmas. We   first assume that no $P_i$ is virtually cyclic, postponing the general case to the next subsection.

\begin{lem}\label{lem:concl}
Consider the graph of groups $\Gamma'$ constructed in Subsection \ref{sec:rid}, and  assume, additionally, that the subgroups $P_1,\dots,P_k$ are  residually finite
(and not virtually cyclic).
Then $\Gamma'$
satisfies
Assumption \ref{df:gr-gr}. In particular, its fundamental group $G'$ is residually finite (by Proposition \ref{prop:G-rf}).
\end{lem}

\begin{proof}

Recall that $\Gamma'$ is bipartite, with $V_1=V_E$ and $V_2=V_R\cup V_{QH}$.
Vertex groups in $V_1$ are elementary, hence residually finite by   assumption.

For $v\in V_R$, the vertex group $G'_v$  is obtained by   amalgamating     $\Z\times C$ and the incident edge groups $G_e$ over a finite group $C$. Define $S=C\setminus\{1\}$ as the obstacle set.
As in Remark~\ref{rem:gr_gps_w_f_edges-ceg}, each finitary filling of $G'_v$ with respect to the incident edge groups is an amalgam over $C$, with factors  being   finite or
 $\Z\times C$.

For QH vertices, residual finiteness (indeed, conjugacy separability) of finitary fillings follows from Lemma \ref{lem:f-by-o_fil-cs}.
The assumption that no $P_i$ is virtually cyclic guarantees that $\{C_1,\dots,C_s\}$ is (up to conjugacy) the family of    incident edge groups (see Subsection \ref{JSJ}).

 Relative hyperbolicity follows from  standard combination theorems for relatively hyperbolic groups (cf. \cite{Dah,Osin-comb}) because vertex groups in $V_2$ are hyperbolic relative to incident edge groups:
 this was pointed out in the proof of  Lemma \ref{lem:f-by-o_fil-cs} in the QH case, and in the rigid case this is a consequence of Definition \ref{def:rh} (as the graph $\mathcal{K}$ one can take the
Bass-Serre tree associated to the splitting of $G'_v$, $v \in V_R$, as an amalgam over $C$ discussed above).
\end{proof}

\begin{lem}
\label{lem:outnormsub}
Suppose that $G$ is a finitely generated group, and $N$ is a centerless normal
subgroup of finite index in $G$.
\begin{enumerate}
\item Some finite index subgroup
$\Out_0(G )\leqslant\Out(G )$ is isomorphic to a quotient of a subgroup of $\Out(N )$ by
a finite normal subgroup $L$.
\item
Let $\mathcal{P}$ be a finite  family of subgroups in $G$ and let $\mathcal{Q}$ be a collection of representatives
of $N$-conjugacy classes among $\{N \cap gHg^{-1} \mid H \in \mathcal{P}, g \in G\}$.
Then some finite index subgroup of
$\Out(G;\mathcal{P})$ is isomorphic to a quotient of a subgroup of $\Out(N;\mathcal{Q})$ by
a finite normal subgroup. In particular, if ${\rm Out}(N;\mathcal{Q})$ is residually finite
then so is $\Out(G;\mathcal{P})$.
\end{enumerate}
\end{lem}

\begin{proof} The first assertion is standard (see for instance \cite[Lemma 5.4]{G-L-i}). One defines     $\Aut_0(G)$ as the set    of automorphisms   mapping $N$ to itself and acting as the identity on $G/N$,
and $\Out_0(G)$ is its image in $\Out(G)$.   Using the fact that $N$ is centerless, one shows that the natural map $\Aut_0(G)\to \Aut(N)$ is injective.
The group
  $L$ comes from the action of inner automorphisms of $G$.

For 2, observe that    automorphisms in $ \Aut_0(G)$ preserving the set  of conjugacy classes of   groups in $\calp$ also preserve the (finite) set of $N$-conjugacy classes of   subgroups from $\mathcal Q$.
\end{proof}

We can now prove Theorem \ref{gener}   in the case when no peripheral subgroup is virtually cyclic:
 \begin{thm}
\label{thm:gener-non-vc}
Let $G$ be a group hyperbolic relative to  a family of proper finitely generated subgroups $\calp=\{P_1,\dots,P_k\}$.
If $G$ is one-ended relative to $\calp$,   no $P_i$ is virtually cyclic,
and every $P_i$ is residually finite, then $\Out(G; \calp)$ is residually finite.
 \end{thm}

\begin{proof}
Consider  the canonical elementary JSJ tree  $T$ of $G$ relative to the family $\calp$, as in Subsection~\ref{JSJ}.
If $T$   is trivial (a single vertex), then either $G$ is rigid or it   maps onto the fundamental group of a closed hyperbolic 2-orbifold
with finite kernel.
In the former case $\Out(G;\calp)$ is finite (see \cite[Theorem 3.9]{GL6}) and in the latter case $G$ contains a surface subgroup
of finite index (see for instance \cite[Thm.\ 4.3]{Martino}). Therefore in this case $\Out(G)$ is residually finite by a combination of Grossman's theorem \cite[Thm.\  3]{Grossman} with Lemma \ref{lem:outnormsub}.

Hence we can further assume that the canonical JSJ tree $T$  is non-trivial.  Let us apply the construction of Section \ref{sec:rid}.
By Lemma \ref{rid}, a finite index subgroup $\Out^r(G)$ of   ${\Out}(G;\mathcal{P})$ embeds into ${\rm Out}(G')$, where $G'$ is the fundamental group of the bipartite graph of groups $\Gamma'$.

If $T$ has at least one rigid vertex, then the group $G'$ has infinitely many ends, because of the way we  constructed $\Gamma'$. Furthermore, $G'$ is residually finite by Lemma \ref{lem:concl}.
Therefore ${\rm Out}(G')$ is residually finite by \cite[Thm.\  1.5]{M-O-norm_aut},
so its subgroup $\Out^r(G)$, and also ${\Out}(G;\mathcal{P})$,  are residually finite.

Hence we can suppose that the JSJ decomposition
of $G$ has no rigid vertices. In this case $G'=G$ by construction, and $V_2=V_{QH}$.

There are two cases.
Assume, at first, that $E(G)=\{1\}$.
Then, according to Proposition \ref{lem:f-by-o_fil-cs}, the graph of groups $\Gamma'$ satisfies all the assumptions of Proposition \ref{prop:Out-rf}, which
allows us to conclude that ${\rm Out}(G)$ (and, hence, ${\rm Out}(G;\mathcal{P})$) is residually finite.

If  $E(G) \neq \{1\}$, we shall deduce residual finiteness of ${\Out}(G;\calp)$ from Lemma \ref{lem:outnormsub}. As $E(G)$ is finite and  $G$ is residually finite (by    Lemma \ref{lem:concl}),  there exists a finite index normal subgroup
$N\lhd G$ such that $N\cap E(G)=\{1\}$.
It is standard that $N$ is hyperbolic relative to   the family $\calq$ described in the second part of Lemma \ref{lem:outnormsub} (this follows immediately from  Definition \ref{def:rh}).
Note that groups in $\calq$ are finitely generated, residually finite,  not virtually cyclic, and they are proper subgroups of $N$ since groups in $\calp$ have infinite index in $G$
(indeed, each $P_i\in \calp$ is almost malnormal in $G$ -- see \cite[Thm.\  1.4]{Osin-RHG}).

The group $E(N)$ is trivial because it is characteristic in $N$, hence it is  contained in $E(G)\cap N=\{1\}$. In particular, $N$ is centerless.
Moreover, as pointed out in \cite{GLMcCool}, it follows from Theorem IV.1.3 of \cite{DD} that $N$ is one-ended relative to $\calq$.
Since $E(N)=\{1\}$, we know that $\Out(N;\mathcal{Q})$  is residually finite by the previous case, and Lemma \ref{lem:outnormsub} implies that ${\Out}(G;\calp)$ is residually finite.
\end{proof}

The arguments given above show the following facts, which may be of independent interest:

\begin{cor}\label{cor:virt_out_emb}
Let $G$ be a group hyperbolic relative to a   family $\calp=\{P_1,\dots,P_k\}$ of   proper finitely generated residually finite groups,
such that no $P_i$ is virtually cyclic. Suppose that $G$ is one-ended relative to $\calp$.

If the canonical JSJ decomposition of $G$ over elementary subgroups relative to $\calp$ has no rigid vertices, then $G$ is residually finite. Otherwise, a finite index subgroup of $\Out (G;\calp)$ embeds into
$\Out(G')$, where $G'$ is a finitely generated residually finite group with infinitely many ends.

In all cases, $\Out (G;\calp)$ virtually embeds into $\Out(G')$, where $G'$ is a finitely generated
 residually finite relatively hyperbolic group. \qed
\end{cor}

\subsection{Allowing virtually cyclic $P_i$'s} \label{vcpi}

We now prove Theorem \ref{gener}   in general, allowing virtually cyclic peripheral subgroups. We may assume that   all $P_i$'s are infinite, since removing finite groups from
$\calp$ does not affect relative one-endedness.

The new phenomenon occurs at QH vertices of the canonical JSJ decomposition $\Gamma$. With the notations of Subsection \ref{qhg},
it is still true that incident edge groups of $G_v$ are preimages of fundamental groups of boundary components of $\mathcal{O}$,
but there may now be boundary components $\calc_j$ such that $C_j=\xi\m(B_j)$ is not an incident edge group, but is conjugate to a group in $\calp$.

In Section \ref{sec:rid} we used groups $PMCG(G_v)$ and $PMCG^\partial(G_v)$, defined using incident edge groups of $\Gamma$.
We must now  replace $PMCG(G_v)$ with a subgroup, by requiring that automorphisms act on groups $C_j$  as above as inner automorphisms $\tau_{a_j}$ of $G_v$. The group $PMCG^\partial(G_v)$ is
replaced by the preimage of this subgroup under $\pi_v$. We do not keep track of the $a_j$'s; the corresponding $\calc_j$'s should be thought of as punctures rather than boundary components,
in particular there is no twist near them (in the context of Subsection \ref{surf}, isotopies are free on the components $\calc_j$).

With this modification, all arguments given in  Section \ref{sec:rid}, \ref{tfrh} and \ref{gtor} go through. In Proposition \ref{lem:f-by-o_fil-cs},
we define $\calh$ using only the groups $C_j$ which are   incident edge groups.
The hyperbolic orbifold $\O'$
may then have a non-empty boundary.
In this case its fundamental group is virtually free, hence conjugacy separable by \cite{Dyer}.

\subsection{Proof of Corollary \ref{main1}}
\label {pfcor}
  Suppose that $G$ is hyperbolic relative to a family $\calp=\{P_1,\dots,P_k\}$ of virtually polycyclic groups.
Without loss of generality we may assume that no $P_i$ is virtually cyclic (see Subsection \ref{defrhyp}).
The result is true if $G$ is virtually polycyclic \cite{Wehr}. Otherwise, Theorem \ref{gener} applies since every $P_i$ is residually finite.
To conclude, note that $\Out(G;\calp)$ has finite index in $\Out(G)$ because groups in $\calp$ are characterized (up to conjugacy) as maximal virtually polycyclic subgroups which are not virtually cyclic
(see \cite[Lemma 3.2]{M-O-fixed_sbgp} for a more general result).

\section{Residual $p$-finiteness for automorphism   groups
} \label{rep}

\subsection{Residual $p$-finiteness}

Given a prime $p$ and a group $G$,   we will say that a subgroup $K \leqslant G$ has \emph{$p$-power index} in $G$ if $|G:K|=p^k$ for some $k\ge 0$.
\begin{rem}\label{car}
The intersection of two subgroups $H_1,H_2$ of $p$-power index is not necessarily of  $p$-power index,   but it is if $H_1$ and $H_2$ are normal
(for then there is an embedding of $G/(H_1\cap H_2)$ into $G/H_1\times G/H_2$). In particular, if $G$ is finitely generated,  any normal  subgroup
$H$ of $p$-power index contains one which is  characteristic in $G$, namely  the intersection of all subgroups of $G$ with the same index as $H$.
\end{rem}

The collection of normal subgroups of $G$ of $p$-power index forms a basis of neighborhoods of the identity in $G$, giving rise to the \emph{pro-$p$ topology} on $G$.
As in the case of residually finite groups, the pro-$p$ topology on $G$ is Hausdorff if and only if $G$ is \emph{residually $p$-finite}: given $g\ne 1$, there is a homomorphism
$\varphi$ from $G$ to a finite $p$-group such that $\varphi(g)\ne 1$ (by Remark \ref{car}, one may assume that $\ker \varphi$ is characteristic if $G$ is finitely generated).

Residual $p$-finiteness is a much more delicate condition than residual finiteness. It is still
clearly stable under direct products, but in general it is not stable under semidirect products  ($\Z$ is residually $p$-finite for any prime $p$, but it is easily checked that
the Klein bottle group, the non-trivial semidirect product $\Z \rtimes \Z$, is not residually $p$-finite if $p>2$). Even more strikingly, for any given set of prime numbers $\Pi$, there exists a $3$-generated
center-by-metabelian group which is residually $p$-finite if and only if $p \in \Pi$ -- see \cite{Hartley}.

\begin{lem}\label{lem:resp-sbgp} If $A$ has a residually $p$-finite normal subgroup $B$ of $p$-power index, then $A$ is residually $p$-finite.
\end{lem}

\begin{proof}
Take any $x \in A \setminus\{1\}$. Since $B$ is residually $p$-finite, there exists a $p$-power index normal subgroup $N \lhd B$ such that $x \notin N$. The intersection $H$ of all $A$-conjugates of $N$ is
  normal in $A$ and has $p$-power index in $B$ (see Remark \ref{car}), so
$|A:H|=|A:B| |B:H|$ is a power of $p$. Since    $x \notin H$, we can conclude that $A$ is residually $p$-finite.
\end{proof}

\begin{rem} Combining Lemma \ref{lem:resp-sbgp} with an induction on the subnormal index, one can actually prove that any group containing a subnormal residually $p$-finite
subgroup of $p$-power index is itself residually $p$-finite.

It is not difficult to see that not every $p$-power index subgroup of a group $G$ has to be closed in the pro-$p$ topology.
In fact, a $p$-power index subgroup $K \leqslant G$ is closed in the pro-$p$ topology on $G$ if and only if $K$ is subnormal in $G$ (cf. \cite[Lemma A.1]{Toinet}).
\end{rem}

\begin{lem} \label{Gsurfini}
If $G$ is residually $p$-finite, and $N \lhd G$ is a finite normal subgroup, then $G/N$ is also residually $p$-finite.
\end{lem}

\begin{proof}
Indeed, since $G$ is residually $p$-finite, any finite subset of $G$ is closed in the pro-$p$ topology on $G$. Therefore $N$ is the intersection of $p$-power index normal subgroups of $G$, and so
$G/N$ is residually $p$-finite.
\end{proof}

For any prime $p$ and any group $H$, let $\Aut_p(H)$ be the subgroup of $\Aut(H)$ which consists of automorphisms that act trivially on the first mod-$p$ homology  of $H$.
Namely, let $K_p:=[H,H]H^p$ be the verbal subgroup of $H$, which is  the product of the derived subgroup $[H,H]$ and the subgroup $H^p$ generated by all the $p$-th powers of elements in $H$. Then
$$\Aut_p(H)=\{\alpha \in\Aut(H) \mid \alpha(hK_p)=hK_p  \mbox{ for all } h \in H\}.$$
If $H$ is finitely generated, then $K_p$  has finite index in $H$, therefore $\Aut_p(H)$ will have finite index in $\Aut(H)$.
Observe also that all inner automorphisms are in $\Aut_p(H)$ because $H/K_p$ is abelian, and
the group $\Out_p(H):=\Aut_p(H)/\Inn(H)$ has finite index in $\Out(H)$.

The following classical theorem of P. Hall will be useful (cf. \cite[5.3.2, 5.3.3]{Robinson}):

\begin{lem}\label{lem:Hall} If $H$ is a finite $p$-group, then $Aut_p(H)$ is also a finite $p$-group. \qed
\end{lem}

The next statement was originally proved by L. Paris in \cite[Thm 2.4]{Paris}. We   present an elementary proof based on Lemma \ref{lem:Hall}.
\begin{lem}\label{lem:aut_p-rpf} Let $H$ be a finitely generated residually $p$-finite group, for some prime $p$. Then $\Aut_p(H)$ is residually $p$-finite, hence $\Aut(H)$  is virtually residually $p$-finite.
\end{lem}

\begin{proof} Consider any non-trivial automorphism $\alpha \in \Aut_p(H)$. Then there is $h_0 \in H$ such that $\alpha(h_0) \neq h_0$. Since $H$ is residually $p$-finite, there exist a finite $p$-group $K$ and
an epimorphism $\psi: H \to K$ with $\psi(\alpha(h_0)) \neq \psi(h_0)$ in $K$.
As explained in Remark \ref{car}, one may assume that $\ker \psi$ is a characteristic subgroup of $H$.
This implies that $\psi$ naturally induces a homomorphism $\varphi: \Aut(H) \to \Aut(K)$.
Clearly, $\varphi(\Aut_p(H)) \subseteq \Aut_p(K)$, so the restriction
$\varphi'$ of $\varphi$ to $\Aut_p(H)$ gives a homomorphism from $\Aut_p(H)$ to $\Aut_p(K)$, where the latter is a finite $p$-group by Lemma \ref{lem:Hall}. It remains to observe that
$\varphi'(\alpha)$ is non-trivial, because $\varphi'(\alpha)(\psi(h_0))=\psi(\alpha(h_0))\neq \psi(h_0)$ by construction.
\end{proof}

\subsection{Toral  relatively hyperbolic groups}

  In this section we will prove Theorem  \ref{resp}. The method is similar to the one used in Section~\ref{hyp}.

Given a group $H$ with a fixed family of peripheral subgroups $C_1,\dots,C_s$, $s \ge 1$, we can define $\Aut^\partial(H)$, ${PMCG}^\partial(H)$ and ${PMCG}(H)$ as in Subsection \ref{relaut}.
For any prime $p$, let $\Aut^\partial_p (H)\leqslant \Aut^\partial(H)$ consist only of those t-uples $(\alpha;a_1,\dots,a_s)$ for which $\alpha \in \Aut_p(H)$. In other words $\Aut_p^\partial(H)$ is the full preimage
of $\Aut_p(H)$ under the natural projection $\Aut^\partial(H) \to \Aut(H)$. We also define $PMCG_p^\partial(H)$ as the image of $\Aut_p^\partial(H)$ in ${PMCG}^\partial(H)$, and  $PMCG_p(H)$ will denote its
image in $PMCG(H)\leqslant \Out(H)$.

\begin{rem}\label{rem:fg->pmcg_p-fi}
If $H$ is finitely generated, then $PMCG_p^\partial(H)$ has finite index in ${PMCG}^\partial(H)$.
\end{rem}

\begin{lem}\label{lem:fg+resp->pmsg_p-resp}
If $H$ is a finite $p$-group, then so is $PMCG_p^\partial(H)$.
 If $H$ is a finitely generated residually $p$-finite group, then $PMCG_p^\partial(H)$ is residually $p$-finite.
\end{lem}

\begin{proof}
$PMCG_p^\partial(H)$ embeds into $H^{s-1}\rtimes\Aut_p(H)$ (see Remark \ref{rem:emb_of_aut}); this group is a finite $p$-group by Lemma \ref{lem:Hall}. For the second assertion, argue as in Lemma \ref{rfd}, mapping
$PMCG_p^\partial(H)$ to $PMCG_p^\partial(H/N)$ with $H/N$ a finite $p$-group.
\end{proof}

  Our next goal is Lemma \ref{lem:omcg2} below, which is an analogue of Lemma \ref{omcg}. We need to prove two auxiliary statements first.

\begin{lem}\label{lem:surf_gps-rp} Fundamental groups of closed hyperbolic surfaces are residually $p$-finite for all primes $p$.
\end{lem}

\begin{proof} Let $\Sigma$ be a closed hyperbolic surface. Then $\pi_1(\Sigma)$ is residually free, except for the case when $\Sigma=\Sigma_{-1}$ is the closed non-orientable surface  of genus $3$ (and Euler characteristic $-1$), see \cite{GBaumslag-62, BBaumslag-67}. Since free groups are residually $p$-finite for every prime $p$ (cf.\  \cite[6.1.9]{Robinson}), the lemma
 follows for all $\Sigma \neq \Sigma_{-1}$.
On the other hand, for any prime $p$,  there is    a normal cover of degree $p$ of $\Sigma_{-1}$ (because the abelianization of $\pi_1(\Sigma_{-1})$ is isomorphic to $\mathbb{Z}^2 \times \Z/2\Z$). This cover
is a surface of higher genus,   so its fundamental group is residually $p$-finite by the previous argument. Hence $\pi_1(\Sigma_{-1})$ is residually $p$-finite by Lemma \ref{lem:resp-sbgp}.
\end{proof}

\begin{lem}\label{lem:O_n-resp} Let $p$ be a prime, and $n$ be a power of $p$. Let $\Sigma_n$ be a closed hyperbolic 2-orbifold   whose singularities are cone points of order $n$.
Then $O_n:=\pi_1(\Sigma_n)$ is residually $p$-finite.
\end{lem}

\begin{proof}

Let $\Sigma$ be a compact surface obtained by removing a neighborhood of each conical point.
We may map $\pi_1(\Sigma)$ to a finite $p$-group so that the fundamental group of every boundary component has image of order exactly $n$:
if $\Sigma$ has only one boundary component, this
follows from \cite[Lemma 1]{Ste} (see also   \cite[Lemma 4.1]{Martino}); if there are more, the fundamental group of each boundary component is a free generator of  $\pi_1(\Sigma)$, and we
map $\pi_1(\Sigma)$ to $H_1(\pi_1(\Sigma),\Z/n\Z)$, its abelianization mod $n$.

The corresponding normal covering of $\Sigma$ extends to a covering of $\Sigma_n$ by a closed surface,
because its restriction to every component of $\partial \Sigma$ has degree exactly $n$.
The fundamental group of this surface  is residually $p$-finite by Lemma \ref {lem:surf_gps-rp}. Its index in  $O_n$ is a power of $p$, so $O_n$ is  residually $p$-finite by Lemma \ref{lem:resp-sbgp}.
\end{proof}

Now suppose, as in Subsection \ref{surf},   that $H$ is the fundamental group of a compact  surface $\Sigma$ with negative Euler characteristic and $s\ge 1$ boundary components.
Let $C_1,\dots, C_s$ be the fundamental groups of these components, considered as subgroups of $H$.
Let $\calt_H \leqslant PMCG^\partial(H)$ be the corresponding group of twists. Note that $\calt_H \subseteq PMCG_p^\partial(H)$ for any prime $p$.
We have the following analogue of Lemma \ref{omcg}:

\begin{lem}\label{lem:omcg2} Let $p$ be a prime. Then the quotient $PMCG_p^\partial(H) /n\calt_H$ is residually $p$-finite for every sufficiently large power $n$ of $p$.
\end{lem}

\begin{proof}
The   proof is   similar to that of Lemma \ref{omcg},  using $PMCG_p$ instead of $PMCG$.
The kernel of $\theta:\calt_H\to\calt_{O_n}$ is $n\calt_H$,
 and we need to know that $PMCG_p(H)$ and $PMCG_p^\partial(O_n)$ are residually $p$-finite.

Clearly  $PMCG_p(H) \leqslant \Out_p(H)$, which is residually $p$-finite by a result of L. Paris \cite[Thm.\ 1.4]{Paris} since $H$ is a free group.
  On the other hand, the group $PMCG_p^\partial(O_n)$ is residually $p$-finite by Lemmas \ref{lem:O_n-resp} and \ref{lem:fg+resp->pmsg_p-resp}.
\end{proof}

We also need to consider abelian groups.
Let $A$ be a free abelian group of finite rank with a chosen family of subgroups $C_1,\dots,C_s$. For any prime $p$, consider the corresponding groups $\Aut_p^\partial(A)$, $PMCG_p^\partial(A)$, and the normal subgroup
$\calt_A\lhd PMCG_p^\partial(A)$, defined as in Subsection \ref{relaut}. Note that $\calt_A$ is naturally isomorphic to the quotient of $A^s$ by its diagonal subgroup, hence to $A^{s-1}$.

\begin{lem}\label{lem:omcg3} Let $p$ be a prime. The quotient $PMCG_p^\partial(A) /n\calt_A$ is residually $p$-finite for every power $n$ of $p$.
\end{lem}

\begin{proof}
Consider the following commutative diagram of short exact sequences:
 $$
 \xymatrix{
\{1\}\ar[r]&\calt_A\ar[r]\ar[d]^\theta&PMCG_p^\partial(A)\ar[r]\ar[d]& PMCG_p (A)\ar[r]\ar[d]&\{1\}\\
 \{1\}\ar[r]& \calt _{A/nA}\ar[r]& PMCG_p^\partial(A/nA)\ar[r]& PMCG_p (A/nA)\ar[r]& \{1\}.
}
$$
The map
$\theta:\calt_A\to\calt_{A/nA}$ sends $A^{s-1}$ to $(A/nA)^{s-1}$, so its kernel is $n\calt_A$ and the proof is reduced to showing that $PMCG_p(A) $ and $PMCG_p^\partial(A/nA)$ are residually $p$-finite.
Now  $PMCG_p(A) \leqslant \Out_p(A)=\Aut_p(A)$ is residually $p$-finite by Lemma \ref{lem:aut_p-rpf}, and
$PMCG_p^\partial(A/nA)$ is a finite $p$-group  by  Lemma \ref{lem:fg+resp->pmsg_p-resp}.
\end{proof}

\begin{lem}[cf.\  Lemma \ref{alg}]\label{alg2}
  Consider a finite set $V$ and groups $P_v$, $v \in V$, with normal subgroups $T_v$ free abelian of finite ranks. Let $P=\prod_{v \in V} P_v$ and $  T=\prod_{v \in V} T_v \leqslant P$ be their direct products.

  Suppose that $p$ is a prime number and $Z \leqslant T$ is a subgroup such that $T/Z$ contains no $q$-torsion if $q\ne p$ is a prime.
If $P_v/nT_v$ is residually $p$-finite  for all  $v \in V$ and for every sufficiently large power $n$  of $p$, then $Z$ is closed in the pro-$p$ topology of $P$.
In particular, if $Z$ is normal in $P$, then $P/Z$ is residually $p$-finite.
 \end{lem}

 \begin{proof}
This is similar to the proof of Lemma \ref{alg}.  One first proves the result when $Z$ has $p$-power index in $T$. In the general case,
$T/Z$ being residually $p$-finite  guarantees that $Z$ is the intersection of (normal) subgroups of $p$-power index in $T$.
 \end{proof}

We are now ready to prove the main theorem of this section.

\begin{thm}\label{thm:virt_rp} If some finite index subgroup of $G$ is a   one-ended  toral relatively hyperbolic group,
then $\Out(G)$ is virtually residually $p$-finite for any prime $p$.
\end{thm}

\begin{proof}
First suppose that $G$ itself is   torsion-free and hyperbolic relative to a family $\calp=\{P_1,\dots,P_k\}$ of free abelian groups of finite rank. As in the proof of Corollary \ref{main1},
we can assume that no $P_i\in\calp$ is cyclic,  and restrict to $\Out(G;\calp)$ because it has finite index in $\Out(G)$. Consider the canonical  JSJ tree $T$  relative to $\calp$ over abelian groups as in Subsection \ref{JSJ}.

If $T$ consists of a single point then either $G$ is rigid, or $G$ is the fundamental group of a closed hyperbolic surface $\Sigma$, or $G$ is a finitely generated free abelian
group.
In the first case  $\Out(G)$
is finite (see  \cite {GL6} for instance).
In the second case, if $\Sigma$ is orientable then $\Out_p(G)$ is residually $p$-finite by \cite[Thm.\ 1.4]{Paris},
and if $\Sigma$ is non-orientable, then it possesses an orientable cover $\Sigma'$ of degree $2$. Since the group $\Out(\pi_1(\Sigma'))$ is virtually residually $p$-finite by \cite[Thm.\ 1.4]{Paris},
and $\pi_1(\Sigma')$  is a centerless normal subgroup of finite index of $G$,
we can use Lemmas \ref{lem:outnormsub} and \ref{Gsurfini}  to conclude that $\Out(G)$ is virtually residually $p$-finite. Finally, if $G$ is a free abelian group of finite rank,
then $\Out_p(G)=\Aut_p(G)$ is residually $p$-finite by Lemma \ref{lem:aut_p-rpf}.

Thus we can suppose that the tree $T$ is non-trivial. In this case we know (cf. Lemma~\ref{reduc} and Remark \ref{rem:fg->pmcg_p-fi}) that
$\Out_p^r(G)$, the image by $\lambda_{E,QH}$ of   $\prod _{v\in V_E\cup V_{QH}}PMCG_p^\partial(G_v)$, has finite index in $\Out(G;\calp)$.
We apply Lemma \ref{alg2} to $P=\prod _{v\in V_E\cup V_{QH}}PMCG_p^\partial(G_v)$, with $T_v=\calt_v$.

We  know that each $PMCG_p^\partial(G_v)/n\calt_v$ is residually $p$-finite for $n$ a large power of $p$,   by
Lemmas \ref{lem:omcg2} and \ref{lem:omcg3}. The quotient of $\prod _{v\in V_E\cup V_{QH}}\calt_v$ by $Z=\ker\lambda_{E,QH}$ is the whole group of twists $\calt$ by Lemma \ref{tpre}, it is torsion-free (see Corollary 4.4 of \cite{GL6}).
Thus
$\Out_p^r(G)=P/Z$
is residually $p$-finite by   Lemma \ref{alg2}. It has finite index in $\Out(G)$, so $\Out(G)$ is virtually residually $p$-finite.

Now suppose that $G$ contains a toral relatively hyperbolic group $G_0$ with finite index. We may assume that $G_0$ is normal. If $G_0$ is abelian, then $\Out(G)$ is contained in some $GL(n,\Z)$ by \cite{Wehr},
so it is virtually residually $p$-finite by Lemma \ref{lem:aut_p-rpf}   (as $GL(n,\Z)\cong \Aut(\Z^n)$).  Otherwise $G_0$ has trivial center and we apply   Lemmas \ref{lem:outnormsub} and   \ref{Gsurfini}.
 \end{proof}

\subsection{Groups with infinitely many ends}\label{subs:inf_ends}

In this subsection we   prove Theorem \ref{thm:resp-inf_ends}: \emph{if $G$ is a finitely generated group with infinitely many ends, and $G$ is virtually residually $p$-finite for some
prime number $p$, then $\Out(G)$ is virtually residually $p$-finite}.
 The argument will use the following ``pro-$p$'' analogue of Lemma \ref{lem:cs->rf_out}:

\begin{lem}\label{lem:out-resp_crit}
Let $p$ be a prime. Given a finitely generated group $G$ and $\alpha\in\Aut_p(G)$, suppose that there is a homomorphism $\psi:G \to K$ with $K$ a
finite $p$-group such that $\psi(g)$ is not conjugate to $\psi(\alpha(g))$ in $K$. Then there is a
homomorphism $\phi: {\Out_p}(G) \to L$ with $L$ a finite $p$-group such that
$\phi(\hat\alpha) \neq 1$ in $L$, where $\hat \alpha$ denotes the image of $\alpha$ in $\Out_p(G)$.
\end{lem}

\begin{proof} The proof is almost identical to the proof of Lemma \ref{lem:cs->rf_out}, except one needs to use $\Aut_p$ and $\Out_p$ instead of $\Aut$ and $\Out$, together with the fact that
$\Out_p(H)$ is a finite $p$-group for any finite $p$-group $H$, which immediately follows from Lemma \ref{lem:Hall}.
\end{proof}

\begin{proof}[Proof of Theorem \ref{thm:resp-inf_ends}] Recall that by Stallings' theorem for groups with infinitely many ends \cite{Stallings},
the group $G$ splits as an amalgamated product or as an HNN-extension over a finite subgroup $C \leqslant G$.
Since $G$ is virtually residually $p$-finite we can find a finite index normal subgroup $H\lhd G$ such that $H \cap C=\{1\}$ and $H$ is residually $p$-finite.
It follows from the generalized  Kurosh theorem  (cf.\ \cite[I.7.7]{DD}   or \cite[Thm.\  8.27]{Cohen-book})
that $H=A*B$, where $A$ and $B$ are non-trivial
finitely generated residually $p$-finite groups. Note that $H$ has trivial center (as does any non-trivial free product), and so by Lemmas~\ref{lem:outnormsub}
and \ref{Gsurfini} it is enough to prove that $\Out(H)$ is virtually residually $p$-finite.

Observe that $H$ is hyperbolic relative to $\{A,B\}$ and consider any automorphism $\alpha \in \Aut_p(H) \setminus\Inn(H)$. Again, since $H$ splits as a non-trivial free product, $H$ contains no non-trivial finite normal subgroups, hence
$E(H)=\{1\}$. Therefore, according to Lemma~\ref{lem:comm_aut_def}, there exists $g \in H$ such that
both $g$ and $h:=\alpha(g)$ are loxodromic in $H$ and $g$ is not commensurable with $h$ in $H$. Since $A$ and $B$ are residually $p$-finite, applying Lemma \ref{lem:non-comm},
we can find normal subgroups $A'\lhd A$ and $B'\lhd B$ such that
$A_1:=A/A'$ and $B_1:=B/B'$ are finite $p$-groups and the images of $g$ and $h$ are  non-commensurable in the free product $H_1:=A_1*B_1$.

  We claim that $H_1$ is conjugacy $p$-separable, i.e.,  given two non-conjugate $h,h'\in H$ there exist a finite $p$-group $K$ and a homomorphism $\xi:H_1 \to K$ such that
$\xi(h)$ is not conjugate to $\xi(h')$ in $K$. Indeed, by the Kurosh subgroup theorem, the kernel of the natural map $A_1*B_1 \to A_1 \times B_1$ is free,
thus $H_1$ is an extension of a finitely generated free group by the finite $p$-group $A_1 \times B_1$. Hence, by a theorem of E. Toinet \cite[Thm.\   1.7]{Toinet}, $H_1$ is conjugacy $p$-separable
(in fact, the full strength  of Toinet's result is not needed here: conjugacy $p$-separability of free products of finite $p$-groups can be derived from the conjugacy $p$-separability of
the free group via a short argument, similar to the one used by V. Remeslennikov in \cite[Thm.\ 2]{Rem}).

Let $\eta: H \to H_1$ denote the natural homomorphism with
$\ker \eta=\llangle A',B'\rrangle^H$.
Let $\psi:=\xi \circ \eta:H \to K$. Then $\psi(h)=\psi(\alpha(g))$ is not conjugate to $\psi(g)$ in $K$ by construction.
Therefore by Lemma \ref{lem:out-resp_crit} there is a finite $p$-group $L$ and a homomorphism
$\phi:\Out_p(H) \to L$ such that $\phi(\hat\alpha) \neq 1$ in $L$, where $\hat \alpha$
is the image of $\alpha$ in
$\Out_p(H)$. Thus we have shown that $\Out_p(H)$
is residually $p$-finite. Since $H$ is finitely generated, $|\Out(H):\Out_p(H)|<\infty$, and so $\Out(H)$ is virtually residually $p$-finite, as required.
\end{proof}

Before concluding let us discuss one application of Theorem \ref{thm:resp-inf_ends}.
In a recent paper \cite{A-F} Aschenbrenner and Friedl proved that the fundamental group of any compact $3$-manifold $M$ is residually $p$-finite for all but finitely many
primes $p$. Recalling Lubotzky's theorem \cite[Prop.\  2]{Lub},  they derived that $\Aut(\pi_1(M))$ is virtually residually $p$-finite and mentioned that the similar fact for $\Out(\pi_1(M))$
is not yet known.  Theorem~\ref{thm:resp-inf_ends} implies that if a compact orientable $3$-manifold  $M$ is not irreducible, then
$\Out(\pi_1(M))$ is virtually residually $p$-finite (for all but finitely many primes $p$). Indeed, since $M$ is not irreducible, either it is  $\mathbb{S}^2\times\mathbb{S}^1$ or it
decomposes into a connected sum of prime manifolds. In the former case $\pi_1(M) \cong \mathbb{Z}$, and in the latter case $\pi_1(M)$ splits as a non-trivial free product.
Thus either $\Out(\pi_1(M))$ is finite, or $\pi_1(M)$ has infinitely many ends, and so $\Out(\pi_1(M))$ is virtually residually $p$-finite for all but finitely many $p$'s
by Theorem~\ref{thm:resp-inf_ends} (using the result of Aschenbrenner and Friedl \cite{A-F} mentioned above).
Therefore, in order to prove that $\Out(\pi_1(M))$ is virtually residually $p$-finite  for all compact orientable $3$-manifolds $M$, it is enough to consider only irreducible manifolds.

As a finishing remark, one can recall the theorem of Rhemtulla \cite{Rhem} stating that if a group is residually $p$-finite for infinitely many primes $p$, then it is bi-orderable. Unfortunately our methods do not
allow to deduce that $\Out(G)$ has a single finite index subgroup which is residually $p$-finite for infinitely many $p$'s. This is because we rely on Lemma \ref{lem:Hall}, requiring one to pass
to the subgroup $\Out_p(G)$, the index of which generally depends on the prime $p$.

\end{document}